\documentclass[12pt,a4paper]{article}
\usepackage{amsmath, amssymb, amsthm, diagrams}
\usepackage{tikz}
\usepackage{hyperref}

\usepackage{amssymb}

\newcommand{\ot}{\otimes}

\newcommand{\otexp}[2]{{#1^{\ot #2}}}

\makeatletter
\newcommand{\catVec}{
\@ifnextchar[{\bracketCVec}{\nobracketCVec}
}
\def\bracketCVec[#1]{ \mathsf{Vec}_{#1} }
\def\nobracketCVec{ \mathsf{Vec} }
\makeatother

\makeatletter
\newcommand{\downar}{
\@ifnextchar*{\starDownAr}{\nostarDownAr}
}
\def\starDownAr*{ {\raisebox{0.5pt}[0pt][0pt]{\ensuremath{\scriptstyle{\downarrow}}}} }
\def\nostarDownAr{ {\raisebox{1pt}[0pt][0pt]{\ensuremath{\downarrow}}} }
\makeatother

\makeatletter
\newcommand{\upar}{
\@ifnextchar*{\starUpAr}{\nostarUpAr}
}
\def\starUpAr*{ {\raisebox{1pt}[0pt][0pt]{\ensuremath{\scriptstyle{\uparrow}}}} }
\def\nostarUpAr{ {\raisebox{1pt}[0pt][0pt]{\ensuremath{\uparrow}}} }
\makeatother

\makeatletter
\newcommand{\id}{
\@ifnextchar[{\tensorID}{\notensorID}
}
\def\tensorID[#1]{ \otexp{\mathsf{1}}{#1} }
\def\notensorID{ \mathsf{1}	}
\makeatother
\newcommand{\set}[2]{\left\{ #1 \ | \ #2 \right\} }
\newcommand{\dd}{\partial}	
\newcommand{\codd}{\delta}
\newcommand{\dg}[1]{{\left| #1 \right|}}
\newcommand{\into}{\hookrightarrow}
\newcommand{\N}{\mathbb{N}}
\newcommand{\kspan}[1]{\fld\!\left\langle #1 \right\rangle}
\newcommand{\op}{\oplus}
\newcommand{\opw}{\ \oplus\ }
\renewcommand{\Im}{\mathop{\mathrm{Im}}\nolimits}	
\newcommand{\dgmod}[1]{\textrm{dg #1-mod}}
\newcommand{\Der}{\mathop{\mathrm{Der}}\nolimits}
\newcommand{\colim}{\mathop{\mathrm{colim}}}
\newcommand{\Mor}{\mathop{\mathrm{Mor}}\nolimits}
\newcommand{\Ob}{\mathop{\mathrm{Ob}}\nolimits}
\makeatletter
\newcommand{\quism}{
\@ifnextchar[{\quismName}{\quismNoName}
}
\def\quismName[#1]{\xrightarrow[#1]{\sim}}
\def\quismNoName{\xrightarrow{\sim}}
\makeatother
\newcommand{\To}{\xrightarrow}
\newcommand{\oAss}{\mathcal{A}\mathit{ss}}
\newcommand{\oEnd}{{\mathcal{E}\mathit{nd}}}
\makeatletter
\newcommand{\smod}
{
	\@ifnextchar*{\smod@nonsigma}{\smod@sigma}
}
\newcommand{\smod@nonsigma}[1]	%argument je hvezdicka, musi se nacist, aby pak \@ifnextchar kontrolovalo [
{
	\@ifnextchar[{\smod@nonsigma@col}{\smod@nonsigma@uncol}
}
\newcommand{\smod@sigma}
{
	\@ifnextchar[{\smod@sigma@col}{\smod@sigma@uncol}
}
\def\smod@sigma@col[#1]{\mbox{\hspace{-1.5ex}\ensuremath{#1}-\ensuremath{\Sigma}-module}}
\def\smod@nonsigma@col[#1]{\mbox{\hspace{-1.5ex}\ensuremath{#1}-collection}}
\def\smod@sigma@uncol{\mbox{\hspace{-1.5ex}\ensuremath{\Sigma}-module }}
\def\smod@nonsigma@uncol{\mbox{\hspace{-1.5ex}collection}}
\makeatother
\newcommand{\colop}[1]{\mbox{\ensuremath{#1}-operad}}
\newcommand{\ar}[1]{{\mathrm{ar}(#1)}}
\makeatletter
\newcommand{\Fr}{
\@ifnextchar[{\Fr@weight}{\Fr@noweight}
}
\def\Fr@weight[#1]#2{\mathbb{F}^{#1}(#2)}
\def\Fr@noweight#1{\mathbb{F}(#1)}
\makeatother
\newcommand{\inp}[1]{I(#1)}
\newcommand{\out}[1]{O(#1)}
\newcommand{\BC}{\Omega\mathrm{B}}	
\newcommand{\oo}{\circ}
\newcommand{\fpr}{\mathop{\mathrm{*}}\nolimits}	
\newcommand{\bigfpr}{\raisebox{-6pt}[0pt][0pt]{\mbox{\Large *}}}
\newcommand{\uv}[1]{``#1''}
\newcommand{\Kdual}[1]{#1^{\textrm{<}}}
\newcommand{\fld}{\ensuremath{k}}
\newcommand{\ucat}{\ensuremath{\mathsf{C}}}	%underlying category
\newcommand{\MorCNoId}{\Mor\ucat-\{\textrm{identities}\}}	%morphisms without identities
\newcommand{\oC}{\mathcal{C}}
\newcommand{\oCres}{\oC_\infty}
\newcommand{\oD}{\mathcal{D}}
\newcommand{\oDres}{\oD_\infty}
\newcommand{\oA}{\mathcal{A}}
\newcommand{\oB}{\mathcal{B}}
\newcommand{\oAres}{\oA_\infty}
\newcommand{\chir}[1]{[ \hskip -.5mm [#1] \hskip -.5mm ]}
\newcommand{\pol}[2]{\mathcal{P}(#1,#2)}
\newcommand{\ideal}[1]{\mathcal{I}^{<#1}}
\newcommand{\oP}{\mathcal{P}}

\newcommand{\filt}{\mathfrak{F}}
\newcommand{\indg}[1]{||#1||}

\makeatletter
\newcommand{\cm}{
\@ifnextchar[{\bracket}{\nobracket}
}
\def\bracket[#1]{ \Delta^{(#1)} }
\def\nobracket{ \Delta }
\makeatother
\newcommand{\PICLabelling}
{
\left( 
\raisebox{-4ex}{
\begin{tikzpicture}[scale=0.2]
\draw (-2,-2) node[below]{$1$} -- (0,0) node[left]{$a$} -- (0,2);
\draw (0,-2) node[below]{$2$} -- (0,0);
\draw (2,-2) node[below]{$3$} -- (0,0);
\end{tikzpicture}
}
\right) \cdot [312] =
\raisebox{-4ex}{
\begin{tikzpicture}[scale=0.2]
\draw (-2,-2) node[below]{$2$} -- (0,0) node[left]{$a$} -- (0,2);
\draw (0,-2) node[below]{$3$} -- (0,0);
\draw (2,-2) node[below]{$1$} -- (0,0);
\end{tikzpicture}
}}

\newcommand{\PICBlockPerm}
{
\begin{tikzpicture}[scale=0.5]
\draw (1,0.5) node{$1$};
\draw (2,0.5) node{$2$};
\draw (3,0.5) node{$3$};
\draw (4,0.5) node{$4$};
\draw (5,0.5) node{$5$};
\draw (6,0.5) node{$6$};
\draw (0.7,0) rectangle (2.3,1);
\draw (2.7,0) rectangle (3.3,1);
\draw (3.7,0) rectangle (6.3,1);
\draw[->] (1,-0.25)--(1,-0.75);
\draw[->] (2,-0.25)--(2,-0.75);
\draw[->] (3,-0.25)--(3,-0.75);
\draw[->] (4,-0.25)--(4,-0.75);
\draw[->] (5,-0.25)--(5,-0.75);
\draw[->] (6,-0.25)--(6,-0.75);
\draw (1,-1.5) node{$3$};
\draw (2,-1.5) node{$4$};
\draw (3,-1.5) node{$5$};
\draw (4,-1.5) node{$6$};
\draw (5,-1.5) node{$1$};
\draw (6,-1.5) node{$2$};
\draw (0.7,-2) rectangle (1.3,-1);
\draw (1.7,-2) rectangle (4.3,-1);
\draw (4.7,-2) rectangle (6.3,-1);
\end{tikzpicture}
}

\newcommand{\PICTreeOne}
{
\raisebox{-5ex}{
\begin{tikzpicture}[scale=0.5]
\draw (0,0)--(0,-1) node[above left]{$a$}--(-2,-2)--(-2,-3) node[above left]{$b_1$}--(-3,-4) node[below] {$2$};
\draw (-2,-3)--(-1,-4) node[below]{$5$};
\draw (0,-1)--(0,-3) node[above left]{$b_2$}--(0,-4) node[below]{$1$};
\draw (0,-1)--(2,-2)--(2,-3) node[above left]{$b_3$}--(1,-4) node[below]{$4$};
\draw (2,-3)--(2,-4) node[below]{$3$};
\draw (2,-3)--(3,-4) node[below]{$6$};
\filldraw (0,-1) circle (3pt);
\filldraw (-2,-3) circle (3pt);
\filldraw (0,-3) circle (3pt);
\filldraw (2,-3) circle (3pt);
\end{tikzpicture}
}}

\newcommand{\PICTreeTwo}
{
\raisebox{-5ex}{
\begin{tikzpicture}[scale=0.5]
\draw (0,0)--(0,-1) node[above left]{$a$}--(-2,-2)--(-2,-3) node[above left]{$b_1\cdot[21]$}--(-3,-4) node[below] {$5$};
\draw (-2,-3)--(-1,-4) node[below]{$2$};
\draw (0,-1)--(0,-3) node[above left]{$b_2$}--(0,-4) node[below]{$1$};
\draw (0,-1)--(2,-2)--(2,-3) node[above right]{$b_3\cdot[312]$}--(1,-4) node[below]{$6$};
\draw (2,-3)--(2,-4) node[below]{$4$};
\draw (2,-3)--(3,-4) node[below]{$3$};
\filldraw (0,-1) circle (3pt);
\filldraw (-2,-3) circle (3pt);
\filldraw (0,-3) circle (3pt);
\filldraw (2,-3) circle (3pt);
\end{tikzpicture}
}}

\newcommand{\PICTreeThree}
{
\raisebox{-5ex}{
\begin{tikzpicture}[scale=0.5]
\draw (0,0)--(0,-1) node[above right]{$a\cdot[231]$}--(-2,-2)--(-2,-3) node[above left]{$b_2$}--(-2,-4) node[below] {$1$};
\draw (0,-1)--(0,-3)node[above left]{$b_3$}--(-1,-4)node[below]{$4$};
\draw (0,-3)--(0,-4)node[below]{$3$};
\draw (0,-3)--(0.66,-4)node[below]{$6$};
\draw (0,-1)--(2,-2)--(2,-3)node[above left]{$b_1$}--(1.33,-4)node[below]{$2$};
\draw (2,-3)--(3,-4)node[below]{$5$};
\filldraw (0,-1) circle (3pt);
\filldraw (-2,-3) circle (3pt);
\filldraw (0,-3) circle (3pt);
\filldraw (2,-3) circle (3pt);
\end{tikzpicture}
}}

\newcommand{\PICcm}
{
\raisebox{-1ex}{
\begin{tikzpicture}[scale=0.2]
\draw (-1,1) -- (0,0) -- (0,-1);
\draw (1,1) -- (0,0);
\end{tikzpicture}
}}

\newcommand{\PICcmMore}
{
\raisebox{-1.75ex}{
\begin{tikzpicture}[scale=0.2]
\draw (-2.5,2.5) -- (0,0) -- (0,-1);
\draw (-0.5,2.5) -- (-1.5,1.5);
\draw (0.5,2.5) -- (-1,1);
\draw (2.5,2.5) -- (0,0);
\filldraw (0,0.5) circle (1pt);
\filldraw (-0.25,0.75) circle (1pt);
\filldraw (-0.5,1) circle (1pt);
\end{tikzpicture}
}}

\newcommand{\PICfmThree}
{
\raisebox{-1.5ex}{
\begin{tikzpicture}[scale=0.2]
\draw (-2,-2) -- (0,0) -- (0,1);
\draw (2,-2) -- (0,0);
\draw[draw=white,very thick,double=black,double distance=0.4pt] (-1,0) -- (1,-2);
\draw (-3,-2) -- (-1,0) -- (-1,1);
\draw[draw=white,very thick,double=black,double distance=0.4pt] (-2,0) -- (-0,-2);
\draw (-4,-2) -- (-2,0) -- (-2,1);
\end{tikzpicture}
}}

\newcommand{\PICm}
{
\raisebox{-1ex}{
\begin{tikzpicture}[scale=0.2]
\draw (-1,-1) -- (0,0) -- (0,1);
\draw (1,-1) -- (0,0);
\end{tikzpicture}
}}

\newcommand{\PICiks}
{
\raisebox{-1.5ex}{
\begin{tikzpicture}[scale=0.2]
\draw (-1,-1) -- (0,0) -- (0,1) -- (-1,2);
\draw (1,-1) -- (0,0);
\draw (1,2) -- (0,1);
\end{tikzpicture}
}}

\newcommand{\PICcross}
{
\raisebox{-2ex}{
\begin{tikzpicture}[scale=0.2]
\draw (0,0)--(0,1)--(2,3)--(2,4);
\draw[draw=white,very thick,double=black,double distance=0.4pt] (2,1)--(0,3);
\draw (2,0)--(2,1)--(3,2)--(2,3);
\draw (0,1)--(-1,2)--(0,3)--(0,4);
\end{tikzpicture}
}}

\newcommand{\PICiksThree}
{
\raisebox{-2ex}{
\begin{tikzpicture}[scale=0.2]
\draw (-1,-1) -- (0,0) -- (0,1) -- (-2,3);
\draw (1,-1) -- (0,0);
\draw (0,1) -- (2,3);
\draw (-1,2) -- (0,3);
\end{tikzpicture}
}}

\newcommand{\PICcrossThreebold}
{
\raisebox{-3ex}{
\begin{tikzpicture}[scale=0.2]
\draw (-0.5,-3.5)--(0,-4)--(0,-5);
\draw[very thick] (-1,-3)--(2,0)--(2,1);
\draw (0,-4)--(4,0)--(4,1);
\draw[draw=white,very thick,double distance=1.2pt,double=black,] (0,0)--(3.5,-3.5);
\draw (3.5,-3.5)--(4,-4);
\draw[draw=white,very thick,double distance=1.2pt,double=black] (2.5,-0.5)--(4,-2)--(3.5,-2.5);
\draw[very thick] (2,0)--(2.5,-0.5);
\draw[very thick] (4,-2)--(3,-3);
\draw[very thick] (0,1)--(0,0)--(-2,-2)--(-0.5,-3.5);
\draw (4,-5)--(4,-4)--(6,-2)--(4,0);
\end{tikzpicture}
}}

\newcommand{\PICinterm}
{
\raisebox{-2.5ex}{
\begin{tikzpicture}[scale=0.2]
\draw[very thick] (0,0)--(0,1)--(2,3)--(2,3.5);
\draw (2,3.5)--(2,5);
\draw[draw=white,very thick,double=black,double distance=1.2pt] (2,1)--(0,3);
\draw[very thick] (2,0)--(2,1)--(3,2)--(2,3);
\draw[very thick] (0,1)--(-1,2)--(0,3)--(0,3.5);
\draw (0,3.5)--(0,4)--(-1,5);
\draw (0,4)--(1,5);
\end{tikzpicture}
}}

\usepackage{amsthm}

\swapnumbers
\theoremstyle{definition}
\newtheorem{theorem}{Theorem}[section]
\newtheorem{lemma}[theorem]{Lemma}
\newtheorem{conjecture}[theorem]{Conjecture}
\newtheorem{example}[theorem]{Example}
\newtheorem{remark}[theorem]{Remark}
\newtheorem{definition}[theorem]{Definition}

\swapnumbers
\theoremstyle{definition}
\newtheorem{assumptions}[theorem]{Assumptions}

\title{On resolutions of diagrams of algebras}
\author{
Martin Doubek\thanks{The author was supported by GA\v{C}R 201/09/H012 and by SVV-2011-263317.},  \\
Charles University, Prague\\
\texttt{martindoubek@seznam.cz}
}
\date{\today}

\begin{document}

\maketitle
\abstract{
We prove a restricted version of a conjecture by M.~Markl made in \cite{HDA} on resolutions of an operad describing diagrams of algebras.
We discuss a particular case related to the Gerstenhaber-Schack diagram cohomology.
}

\section{Introduction}

As explained in \cite{IB}, the operadic cohomology gives a systematic way of constructing cohomology theories for algebras over an operad $\oP$.
The corresponding deformation complex carries an $L_\infty$-structure describing deformations of $\oP$-algebras.
To make this explicit, one has to find a free resolution of $\oP$.

In particular, we can apply this to the coloured operad $\oA_\ucat$ describing a $\ucat$-shaped diagram of $\oA$-algebras.
An important particular case is $\ucat$ consisting of a single morphism.
This is discussed in \cite{HDA},\cite{FMY} and also, indirectly, in the definition of (weak) $A_\infty$ and $L_\infty$ morphisms.
More complicated categories $\ucat$ received very little attention.
In \cite{HDA}, M. Markl discussed examples leading to the notions of homotopy of $\oA$-algebra morphisms and homotopy isomorphism of $\oA$-algebras.
In the end of the paper, a conjecture partially describing resolutions of $\oA_\ucat$ for any $\oA$ and $\ucat$ appears.
In particular, it settles the question of the existence of the minimal resolution of $\oA_\ucat$.
We discuss this conjecture and prove it in the restricted case of $\oA$ being a Koszul operad with generating operations concentrated in a single arity and degree, see the main Theorem~\ref{MAIN}.

The idea is to glue together a \emph{minimal} resolution of $\oA$ and any cofibrant free resolution of $\ucat$.
The generators of the resulting resolution $\oDres$ are described explicitly as well as the principal part of the differential $\dd$.
To state the theorem precisely requires some preliminary work.

First, we discuss operadic resolutions $\oCres$ of categories.
The operads in question are concentrated in arity $1$, hence this is just a \uv{coloured} version of classical homological algebra.
We deal with maps $\chir{-}_n:\oCres\to\otexp{\oCres}{n}$ with certain prescribed properties.
These are needed to construct the principal part of $\dd$.
We show that these maps are induced by certain coproducts on $\oCres$, thus relating them to (coloured) dg bialgebra structures on $\oCres$.

The proof of the main theorem follows the ideas of M. Markl from \cite{HDA}.
It is necessarily more complicated technically and we discuss it in detail in a separate section.
We find it convenient to recall some technical results of coloured operad theory, namely a version of the K\"unneth formula for the composition product $\oo$, which is very useful for homological computations.
Hence we spend some time in the initial part of the paper explaining basics, though we expect the reader is already familiar with coloured operads.

The case $\oCres$ being the bar-cobar resolution is particularly interesting.
Here, $\oCres$ has a topological flavour, it is completely explicit and we even make $\chir{-}$ explicit.
The resulting resolution $\oDres$ conjecturally gives rise to the Gerstenhaber-Schack complex for diagram cohomology \cite{GS}.

Finally, let me thank Martin Markl for many useful discussions.

\pagebreak

\tableofcontents
\bigskip

In \textbf{Section \ref{SectionBasics}}, we briefly recall basic notions of coloured operad theory.
We focus on the interplay between the colours and $\Sigma$ action.
We prove a version of the K\"unneth formula in Section \ref{SectionKunneth}.
It computes the homology of the composition product.

In \textbf{Section \ref{SectionStatement}}, we prepare necessary notions to formulate the main theorem.
In Section \ref{SectionCategory}, we discuss operadic resolutions of categories and give several examples.
In Section \ref{SectionChiral}, we introduce $\chir{-}_n$ maps, certain combinatorial structures on the resolution of the category.
We prove that these maps always exist and recall some examples from the literature.
We show that $\chir{-}_n$'s are induced by $\chir{-}_2$, which is a certain coproduct on the resolution.
In Section \ref{SectionOperad}, we explain how diagrams of algebras are described by coloured operads and show that this construction is functorial and quism-preserving.
Section \ref{SectionTheorem} contains the statement of the main theorem and compares it to the conjecture by M. Markl.

In \textbf{Section \ref{SectionProof}}, the main theorem is proved.
In Section \ref{SectionDiscussion}, we try to explain the structure of the proof and to point out the places where an improvement might be possible.

In \textbf{Section \ref{SectionBC}}, we recall the bar-cobar resolution of the category, then we make $\chir{-}_n$'s explicit by endowing the resolution with a (coloured) bialgebra structure.
Finally, we discuss the conjectural relation to Gerstenhaber-Schack diagram cohomology.

\pagebreak

%%%%%%%%%%%%%%%%%%%
%%%%%%%%%%%%%%%%%%%
\section{Basics} \label{SectionBasics}
%%%%%%%%%%%%%%%%%%%
%%%%%%%%%%%%%%%%%%%

%%%%%%%%%%%%%%%%%%%
%%%%%%%%%%%%%%%%%%%
\subsection{Conventions and reminder}
%%%%%%%%%%%%%%%%%%%
%%%%%%%%%%%%%%%%%%%

We will use the following notations and conventions:
\begin{itemize}
\item $\N_0$ is the set of natural numbers including $0$.
\item $\fld$ is a fixed field of characteristics $0$.
\item $\kspan{S}$ is the $k$-linear span of the set $S$.
\item $\ot$ always means tensor product over $\fld$.
\item $\Sigma_n$ is the permutation group on $n$ elements.
\item $V$ denotes a set (of colours\footnote{$V$ actually stands for Vertices, which will become apparent later.}).
\item $\ar{x}$ is arity of the object $x$, whatever $x$ is.
\item Vector spaces over $k$ are called \emph{$k$-modules}, chain complexes of vector spaces over $k$ with differential of degree $-1$ are called \emph{dg-$k$-modules} and morphisms of chain complexes are called just \emph{maps}.
Chain complexes are assumed \textbf{non-negatively graded} unless stated otherwise.
The degree $n$ summand of \dgmod{\fld} $C$ is denoted $C_n$.
We let $C_{\leq n} := \bigoplus_{0\leq i\leq n}C_i$ and similarly for other inequality symbols.
Similar notation is used e.g. for \smod[V]s of Definition \ref{DEFsmod}.
\item $\upar C$ denotes the \emph{suspension} of the graded object $C$, that is $(\upar C)_n=C_{n-1}$.
Similarly, the \emph{desuspension} is defined by $(\downar C)_n=C_{n+1}$.
\item $\dg{x}$ is the degree of an element $x$ of a dg-$k$-module.
\item $H_*(C)$ is homology of the object $C$, whatever $C$ is.
\item \emph{Quism} is a map $f$ of dg-$k$-modules such that the induced map $H_*(f)$ on homology is an isomorphism.
\end{itemize}

We extend the notation introduced in section Basics of \cite{D} for $V$-coloured non-$\Sigma$ operads to $V$-coloured $\Sigma$-operads.

\begin{definition}
A permutation $\sigma : \{ 1,\ldots,n \}\to\{ 1,\ldots,n \}$ will also be denoted by $[\sigma(1)\sigma(2)\cdots\sigma(n)]$.

Let $S$ be any set.
Let $\vec{s}=(s_1,\ldots,s_n)\in S^n$.
If a context is clear, we may use this vector notation without explanation.
$S^n$ carries a right $\Sigma_n$ action $$\vec{s}\cdot\sigma := (s_{\sigma(1)},\ldots,s_{\sigma(n)}).$$
If $f:\otexp{A}{n}\to A$ is a linear map, the right $\Sigma_n$ action on $f$ is defined by $$(f\cdot\sigma)(\vec{a}):=f(\sigma\cdot\vec{a}):=f(\vec{a}\cdot\sigma^{-1})=f(a_{\sigma^{-1}(1)},\ldots,a_{\sigma^{-1}(n)})$$
for $\vec{a}\in A^n$.
This is useful for intuitive understanding of the right $\Sigma_n$ action on elements of an operad.
While drawing pictures, we use the convention that into a leaf labelled $i$, the $i^\textrm{th}$ input element is inserted.
Hence element $a\cdot \sigma$ is drawn with labels $\sigma^{-1}(1),\sigma^{-1}(2),\ldots,\sigma^{-1}(n)$ from left to right, e.g. $$\PICLabelling.$$

For $\vec{v}\in V^n$, let
$$\Sigma_{\vec{v}} := \set{\sigma\in\Sigma_n}{\vec{v}=\vec{v}\cdot{\sigma}} = \set{\sigma\in\Sigma_n}{v_i=v_{\sigma(i)}\textrm{ for each }1\leq i\leq n}$$
be the stabilizer of $\vec{v}$ under the action of $\Sigma_n$.
\end{definition}

\begin{definition} \label{DEFsmod}
A \textbf{dg \smod[V]} $X$ is a set $$\set{X(n)}{n\in\N_0}$$ of dg right $\kspan{\Sigma_n}$-modules such that each of them decomposes $$X(n)=\bigoplus_{\substack{v\in V,\\v_1,\ldots,v_n\in V}} X\binom{v}{v_1,\ldots,v_n}$$ as a dg $\fld$-module and $\sigma\in\Sigma_n$ acts by a dg $\fld$-module morphism $$\cdot\sigma: X\binom{v}{v_1,\ldots,v_n} \to X\binom{v}{v_{\sigma(1)},\ldots,v_{\sigma(n)}}.$$

It follows that $X\binom{v}{\vec{v}}$ is a dg $\kspan{\Sigma_{\vec{v}}}$-module.
In particular, the differential commutes with the $\kspan{\Sigma_{\vec{v}}}$ action.

A \textbf{dg $V$-$\Sigma$-operad} is a dg \smod[V] with the usual operadic compositions $\oo_i$.
The axioms these compositions satisfy are the same as those for non-$\Sigma$ dg \colop{V} (see \cite{D}, Definition $2.1$) and we moreover require $\oo_i$'s to be equivariant in the usual sense (see \cite{MSS}, Definition 1.16 for noncoloured case).

We usually omit the prefix $\Sigma$ and non-$\Sigma$.
If $a,b$ are elements of an operad $\oA$ and $\ar{a}=1$, we usually abbreviate $ab := a\oo b := a\oo_1 b$.
If $\ar{a}=n$, we also abbreviate $a(b_1\ot\cdots\ot b_n):=(\cdots((a\oo_1 b_1)\oo_2 b_2)\cdots)\oo_n b_n$.
If $V$ is a single element set, we omit the prefix \uv{$V$-}, otherwise we strictly keep the prefix.
\end{definition}

Now we discuss the composition product $\oo$ on the category of \smod[V]s.
We need some preliminary notions first.

\begin{definition} \label{DEFBlockPerm}
Let $l_1,\ldots,l_m$ be nonnegative integers.
For $n:=l_1+\cdots+l_m$, there is the inclusion 
$$\Sigma_{l_1}\times\cdots\times\Sigma_{l_m} \into \Sigma_n$$
given by
$$(\lambda_1\times\cdots\times\lambda_m)(l_1+\cdots+l_{i-1}+j) := l_1+\cdots+l_{i-1}+\lambda_i(j),$$
where $1\leq i\leq m$ and $1\leq j\leq l_i$.
If $l_i=0$, we set $\Sigma_{l_i}=\Sigma_0:=\{\id\}$.

%Denote $$\Sh{l_1,\ldots,l_m}$$ the set of all $(l_1,\ldots,l_m)$-\textbf{shuffles}, that is permutations $\sigma\in\Sigma_n$ such that for every $1\leq i\leq m$ $$\sigma(l_1+\cdots+l_{i-1}+1) < \sigma(l_1+\cdots+l_{i-1}+2) < \cdots <  \sigma(l_1+\cdots+l_i).$$

Let $\tau\in\Sigma_m$.
Denote 
\begin{align*}
\overline{\tau} := \big[ \ &l_1+\cdots+l_{\tau(1)-1}+1,\ldots,l_1+\cdots+l_{\tau(1)},\\
& l_1+\cdots+l_{\tau(2)-1}+1,\ldots,l_1+\cdots+l_{\tau(2)},\\ 
& \ldots,\\
& l_1+\cdots+l_{\tau(m)-1}+1,\ldots,l_1+\cdots+l_{\tau(m)}\big].
\end{align*}
If $l_i=0$, the block $l_1+\cdots+l_{\tau(i)-1}+1,\ldots,l_1+\cdots+l_{\tau(i)}$ is empty and therefore is omitted in the expression above.
Equivalently, the above formula states $$\overline{\tau}(l_{\tau(1)}+l_{\tau(2)}+\cdots+l_{\tau(i-1)}+j) := l_1+l_2+\cdots+l_{\tau(i)-1}+j$$
for any $1\leq i\leq m$ and $1\leq j\leq l_i$.
$\overline{\tau}$ is called $(l_1,\ldots,l_m)$-\textbf{block permutation} corresponding to $\tau$.
\end{definition}

\begin{example}
\begin{itemize}
\item $[21]\times\id\times[312] = [213645]$
\item $(2,1,3)$-block permutation corresponding to $\tau=[231]$ is $\overline{[231]} = [345612]$:
$$\PICBlockPerm$$
\end{itemize}
\end{example}

\begin{definition}
Fix $v\in V$ and $\vec{v}=(v_1,\ldots,v_n)\in V^n$.
Let $\oA=(\oA,\dd_\oA)$, $\oB=(\oB,\dd_\oB)$ be dg \smod[V]s,
let $l_1,\ldots,l_m$ be nonnegative integers such that $l_1+\cdots+l_m = n$.
For each $1\leq i\leq m$, let $\vec{w}_i=(w_{i,1},\ldots,w_{i,l_i}) \in V^{l_i}$.
Denote $\vec{W} = (\vec{w}_1,\ldots,\vec{w}_m) = (w_{1,1},\ldots,w_{m,l_m}) \in V^n$.
Let 
\begin{align} 
\Sigma(\vec{W},\vec{v}) :&= \set{\sigma\in\Sigma_n}{\vec{W}\cdot\sigma = \vec{v}}\ = \label{SSPreservatorSigma} \\
&=\set{\sigma\in\Sigma_n}{w_{i,j}=v_{\sigma^{-1}(l_1+\cdots+l_{i-1}+j)}\textrm{ for every }1\leq i\leq m,\ 1\leq j\leq l_i}. \nonumber
\end{align}

For fixed $l_1,\ldots,l_m$ and $\vec{w}=(w_1,\ldots,w_m)$
$$\bigoplus_{\vec{W}} \oB\binom{w_1}{\vec{w}_1} \ot \cdots \ot \oB\binom{w_m}{\vec{w}_m} \ot \kspan{\Sigma(\vec{W},\vec{v})}$$
is a dg\footnote{$\kspan{\Sigma(\vec{W},\vec{v})}$ is concentrated in degree $0$.} right $\kspan{\Sigma_{l_1}\times\cdots\times\Sigma_{l_m}}$-module via
$$(b_1\ot\cdots\ot b_m \ot \sigma)\cdot (\lambda_1\times\cdots\times\lambda_m) := (b_1\cdot\lambda_1)\ot\cdots\ot (b_m\cdot\lambda_m) \ot (\lambda_1\times\cdots\times\lambda_m)^{-1}\sigma.$$
Denote the space of coinvariants of this $\kspan{\Sigma_{l_1}\times\cdots\times\Sigma_{l_m}}$-module by the lower index $\Sigma_{l_1}\times\cdots\times\Sigma_{l_m}$.

Now assume only $m$ is fixed and consider
$$\bigoplus_{\substack{l_1,\ldots,l_m\\ \vec{w}}} \oA\binom{v}{\vec{w}} \ot \left( \bigoplus_{\vec{W}} \oB\binom{w_1}{\vec{w}_1} \ot \cdots \ot \oB\binom{w_m}{\vec{w}_m} \ot \kspan{\Sigma(\vec{W},\vec{v})} \right)_{\Sigma_{l_1}\times\cdots\times\Sigma_{l_m}}.$$
This is dg right $\kspan{\Sigma_m}$-module via
$$(a \ot b_1\ot\cdots\ot b_m \ot \sigma)\cdot \tau = (a\cdot\tau) \ot b_{\tau(1)} \ot \cdots \ot b_{\tau(m)} \ot \overline{\tau^{-1}}\sigma,$$
where the bar denotes the corresponding $(l_1,\ldots,l_m)$-block permutation of Definition \ref{DEFBlockPerm}.
It is easy to verify that this action is well defined.

Finally, by taking the $\Sigma_m$ coinvariants and summing over $m$ in the above formula, we get the desired \textbf{composition product} of \smod[V]s:
\begin{gather}
(\oA\oo\oB)\binom{v}{\vec{v}} := \label{EQCompPoductExplicitly} \\
\bigoplus_{m} \left( \bigoplus_{\substack{l_1,\ldots,l_m\\ \vec{w}}} \oA\binom{v}{\vec{w}} \ot \left( \bigoplus_{\vec{W}} \oB\binom{w_1}{\vec{w}_1} \ot \cdots \ot \oB\binom{w_m}{\vec{w}_m} \ot \kspan{\Sigma(\vec{W},\vec{v})} \right)_{\Sigma_{l_1}\times\cdots\times\Sigma_{l_m}} \right)_{\Sigma_m}, \nonumber
\end{gather}
where
\begin{itemize}
\item $m$ runs through nonnegative integers,
\item $l_1,\ldots,l_m$ run through nonnegative integers so that $l_1+\cdots+l_m = n$,
\item $\vec{w}=(w_1,\ldots,w_m)$ runs through $V^m$,
\item $\vec{W}=(\vec{w}_1,\ldots,\vec{w}_m)$ runs through $m$-tuples of $\vec{w}_i$'s, where $\vec{w}_i\in V^{l_i}$,
\item $\Sigma(\vec{W},\vec{v})$ is given by \eqref{SSPreservatorSigma}.
\end{itemize}
To finish the definition of $\oA\oo\oB$, we let $\pi\in\Sigma_n$ act by
$$(a\ot b_1\ot\cdots\ot b_m\ot\sigma)\cdot\pi := a\ot b_1\ot\cdots\ot b_m\ot\sigma\pi.$$

We usually omit the coinvariants from the notation while dealing with elements of $\oA\oo\oB$.
\end{definition}

The purpose of $\Sigma(\vec{W},\vec{v})$ is to label the leaves so that for each $i$, the leaf labelled by $i$ is of colour $v_i$.
The purpose of the coinvariants is the usual one:

\begin{example}
By looking at the pictures, we find that we certainly want the equality
\begin{align*}
a \ot b_1\ot b_2\ot b_3 \ot [251436]^{-1} &= a \ot b_1[21]\ot b_2\ot b_3[312] \ot [521643]^{-1}\\
\PICTreeOne &= \PICTreeTwo.
\end{align*}
But since $[521643] = [251436]([21]\times\id\times[312])$, the above equality is forced by taking the $\Sigma_{l_1}\times\cdots\times\Sigma_{l_m}$ coinvariants.

We also want
\begin{align*}
a \ot b_1\ot b_2\ot b_3 \ot [251436]^{-1} &= a[231] \ot b_2\ot b_3\ot b_1 \ot [143625]^{-1}\\
\PICTreeOne &= \PICTreeThree.
\end{align*}
But $[143625] = [251436]\overline{[231]}$, hence this equality is forced by the $\Sigma_m$ coinvariants.
\end{example}

\begin{definition} \label{DEFMinimal}
Let $X$ be a \smod[V].
The free \colop{V} generated by $X$ carries the \emph{weight grading} $$\Fr{X}=\bigoplus_{i\geq 0}\Fr[i]{X},$$
where $\Fr[i]{X}$ is spanned by free compositions of exactly $i$ generators.
If $X$ is moreover dg \smod[V], the dg structure is inherited to $\Fr{X}$ in the obvious way and we obtain a free dg \colop{V}.
However, $\Fr{X}$ can be equipped with a differential which doesn't come from $X$ and in this case, $(\Fr{X},\dd)$ is called \emph{quasi-free}.

Recall a quasi-free dg \colop{V} $(\Fr{X},\dd)$ is called \textbf{minimal} iff $\Im\dd\subset\Fr[\geq 2]{X}$.
As usual, \emph{free resolution} means a quism $(\Fr{X},\dd)\quism(\oA,\dd)$ with a quasi-free source.
A \emph{minimal resolution} is a resolution with a minimal source.
\end{definition}

%%%%%%%%%%%%%%%%%%%
%%%%%%%%%%%%%%%%%%%
\subsection{A K\"unneth formula} \label{SectionKunneth}
%%%%%%%%%%%%%%%%%%%
%%%%%%%%%%%%%%%%%%%

Our next task is to prove a version of the K\"unneth formula:

\begin{lemma} \label{YetAnotherKunneth}
Let $(\oA,\dd_\oA)$, $(\oB,\dd_\oB)$ be dg \smod[V]s. 
Then there is a graded \smod[V] isomorphism
$$H_*((\oA\oo\oB),\dd) \cong H_*(\oA,\dd_\oA)\oo H_*(\oB,\dd_\oB).$$
\end{lemma}

\begin{proof}
Let $G$ be a finite group, let $(M,\dd)$ be a dg $\kspan{G}$-module.
Obviously, $\dd$ descends to coinvariants, hence $(M_G,\dd)$ is a dg $\kspan{G}$-module too.
We claim
\begin{gather} \label{EQHomologyGroupComm}
H_*(M_G,\dd) \cong (H_*(M,\dd))_G.
\end{gather}
By Maschke's theorem, $$M=\bigoplus_{i\in I} M^i,$$ where $M^i$'s are irreducible $\kspan{G}$-modules.
$\dd$ is $G$-equivariant, hence for each $i$ either $\dd M^i=0$ or $\dd:M^i\To{\cong}M^j$ is an isomorphism for some $j\neq i$.
Denote $$I_P:=\set{i\in I}{\dd M^i=0\textrm{ and there is no }j\textrm{ such that }\dd M^j=M^i}.$$
Also, for each $i$, either $\dd M^i_G=0$ (iff $\dd M^i=0$) or $\dd:M^i_G\To{\cong}M^j_G$ is isomorphism for some $j\neq i$ (iff $\dd:M^i\To{\cong}M^j$).
Then 
\begin{align*}
H_*(M_G,\dd) &= H_*((\bigoplus_{i\in I}M^i)_G,\dd) = H_*(\bigoplus_{i\in I}M^i_G,\dd) \cong \bigoplus_{i\in I_P}M^i_G \cong\\
&\cong (\bigoplus_{i\in I_P}M^i)_G = (H_*(\bigoplus_{i\in I} M^i,\dd))_G \cong (H_*(M,\dd))_G.
\end{align*}
\eqref{EQHomologyGroupComm} is proved.

Let's set some shorthand notation.
In \eqref{EQCompPoductExplicitly}, denote $B(\oB):=\bigoplus_{\vec{W}} \oB\binom{w_1}{\vec{w}_1} \ot \cdots \ot \oB\binom{w_m}{\vec{w}_m} \ot \kspan{\Sigma(\vec{W},\vec{v})}$.
Denote $A(\oA):=\oA\binom{v}{\vec{w}}$ and $\Sigma:=\Sigma_{l_1}\times\cdots\times\Sigma_{l_m}$.
Notice that we are suppressing the dependency on $l_1,\ldots,l_m$ and $\vec{w}$.
Omit $v$ and $\vec{v}$ too.
Hence \eqref{EQCompPoductExplicitly} becomes $$\oA\oo\oB=\bigoplus_m\,\big[ \bigoplus_{\substack{l_1,\ldots,l_m\\ \vec{w}}} A(\oA) \ot B(\oB)_\Sigma \big]_{\Sigma_m}.$$
Let's compute:
\begin{align*}
H_*(\oA\oo\oB,\dd) &= \bigoplus_m H_*\big( \big[ \bigoplus_{\substack{l_1,\ldots,l_m\\ \vec{w}}} A(\oA) \ot B(\oB)_\Sigma \big]_{\Sigma_m}\big) \cong \bigoplus_m \big[ \bigoplus_{\substack{l_1,\ldots,l_m\\ \vec{w}}} H_*(A(\oA) \ot B(\oB)_\Sigma) \big]_{\Sigma_m} \cong\\
&\cong \bigoplus_m \big[ \bigoplus_{\substack{l_1,\ldots,l_m\\ \vec{w}}} H_*(A(\oA)) \ot \big(H_*(B(\oB))\big)_\Sigma \big]_{\Sigma_m} \cong\cdots
\end{align*}
The last isomorphism is provided by the usual K\"unneth formula and \eqref{EQHomologyGroupComm}.
Now trivially $H_*(A(\oA)) = A(H_*(\oA,\dd_\oA))$ and another application of the K\"unneth formula gives $H_*(B(\oB))\cong B(H_*(\oB,\dd_\oB))$ and we finish:
$$\cdots\cong \bigoplus_m \big[ \bigoplus_{\substack{l_1,\ldots,l_m\\ \vec{w}}} A(H_*(\oA,\dd_\oA)) \ot B(H_*(\oB,\dd_\oB))_\Sigma \big]_{\Sigma_m} = H_*(\oA,\dd_\oA) \oo H_*(\oB,\dd_\oB).$$
\end{proof}

%%%%%%%%%%%%%%%%%%%
%%%%%%%%%%%%%%%%%%%
\section{Statement of main theorem} \label{SectionStatement}
%%%%%%%%%%%%%%%%%%%
%%%%%%%%%%%%%%%%%%%

%%%%%%%%%%%%%%%%%%%
%%%%%%%%%%%%%%%%%%%
\subsection{Operadic resolution of category} \label{SectionCategory}
%%%%%%%%%%%%%%%%%%%
%%%%%%%%%%%%%%%%%%%

Let $\ucat$ be a small category and denote $$V:=\Ob\ucat$$ the set of its objects.
For a morphism $f\in\Mor\ucat$, let $\inp{f}$ be its source (Input) and $\out{f}$ its target (Output).
Let $\oC$ be the operadic version of $\oC$, that is
\begin{gather} \label{EQOpeVerCat}
\oC:=\kspan{\Mor\ucat}
\end{gather}
is seen as a coloured \colop{V} concentrated in arity $1$, where each $f\in\Mor\ucat$ is an element of $\oC\binom{\out{f}}{\inp{f}}$ and the operadic composition is induced by the categorical composition.
Obviously, $\oC$ can be presented as $$\oC = \frac{\Fr{\kspan{\MorCNoId}}}{(\textrm{relations})},$$ where each relator is generated by those of the form $r_1-r_2$ with $r_1,r_2$ being operadic compositions of elements of $\Mor\ucat$.
Recall that the elements corresponding to the identities become a part of the free operad construction.

Every such \colop{V} $\oC$ has a free resolution of the form $$\oCres := (\Fr{F},\dd) \quism (\oC,0),$$
where the graded \smod[V]\footnote{Of course, the action of $\Sigma_1$ carries no information and can be omitted.} $F=\bigoplus_{i\geq 0}F_i$ satisfies
\begin{assumptions} \label{CatResCond}
\quad
\begin{enumerate} 
 \item $F_0 = \kspan{M}$ for some $M\subset\MorCNoId$,
 \item $F_1 = \kspan{R}$, where for each $r\in R$, $\dd r=r_1-r_2$ for some free operadic compositions $r_1,r_2$ of elements of $M\cup\{\textrm{identities}\}$.
\end{enumerate}
\end{assumptions}
The existence of such a resolution is quite obvious and we will give several examples below.
A general example is given by the bar-cobar resolution, which will be discussed later in Section \ref{SectionBC} in detail.
Before giving the examples, we note that
\begin{gather} \label{EQNote}
\oC \cong \frac{\Fr{F_0}}{(\dd F_1)}.
\end{gather}

\begin{example}
Let $\ucat$ be the category generated by $2$ distinct morphisms between $3$ distinct objects as in the picture:
\begin{diagram}
V_1 & \rTo^{f} & V_2 & \rTo^g & V_3 
\end{diagram}
Then $\Ob\ucat=V=\{V_1,V_2,V_3\}$, $\Mor\ucat=\{\id_{V_1},\id_{V_2},\id_{V_3},f,g,\,h:=gf\}$.
The composition is obvious.
The \colop{V} $\oC$ has colour decomposition $\oC\binom{V_2}{V_1}=\kspan{f}$, $\oC\binom{V_3}{V_2}=\kspan{g}$, $\oC\binom{V_3}{V_1}=\kspan{h}$.
$\oC$ has the following $2$ obvious resolutions:
\begin{enumerate}
\item Directly from the obvious presentation of $\oC$, we get
$$(\Fr{\kspan{f,g,h,H}},\dd) \quism (\oC,0),$$
where $f,g,h$ are copies of the corresponding generators of $\oC$ and $\inp{H}=V_1$, $\out{H}=V_3$.
The degrees are as follows : $\dg{f}=\dg{g}=\dg{h}=0$ and $\dg{H}=1$.
The differential $\dd$ vanishes on $f,g,h$ and $\dd H = gf-h$.
\item A \uv{smaller} resolution of $\oC$ is
$$(\Fr{\kspan{f,g}},0) \quism (\oC,0).$$
It has less generators because the existence of $h$ is already forced by the existence of $f,g$.
This is an example of a minimal resolution of Definition \ref{DEFMinimal}.
\end{enumerate}
\end{example}

\begin{example} \label{SingleMorphism}
The category
\begin{diagram}
V_1 & \rTo^f & V_2
\end{diagram}
has, apart from the obvious one, a free resolution
$$\oCres := (\Fr{\kspan{f,g,H}},\dd) \quism (\oC,0),$$
where $\inp{g}=\inp{H}=V_1$, $\out{g}=\out{H}=V_2$, $\dg{g}=0$, $\dg{H}=1$ and $\dd H = f-g$.
It was observed in \cite{HDA} that every algebra over $\oCres$ corresponds to a pair of dg \fld-modules, a pair of morphisms $f,g$ between these and a homotopy $H$ between $f$ and $g$.
Hence even resolutions of boring categories, such as $\oC$ in this example, may lead to interesting concepts.
\end{example}

\begin{example}
Probably the simplest example of $\ucat$ which can't be resolved in degrees $0$ and $1$ only is given by the commutative cube:
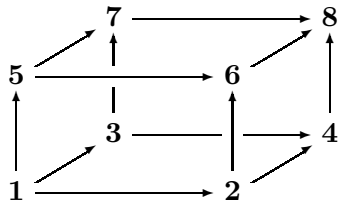
\begin{diagram}[size=1em]
					 &			 & \mathbf{7} &	& \rTo       &			 & \mathbf{8} \\
					 & \ruTo & \uTo  		  &	&            & \ruTo & 						\\					 
\mathbf{5} & \rTo	 & \HonV		  &	& \mathbf{6} &			 & \uTo   		\\
					 &			 & \vLine		  &	&            &			 & 						\\
\uTo			 &			 & \mathbf{3} & \hLine & \VonH      & \rTo     & \mathbf{4}	\\
           & \ruTo &            & & \uTo   	 & \ruTo &						\\
\mathbf{1} &       & \rTo       & & \mathbf{2} & 			 &
\end{diagram}
Objects (i.e. vertices) are denoted $\mathbf{1},\cdots,\mathbf{8}$,
edges (and the corresponding generators of the resolution below) are denoted $(ab)$ with $\mathbf{8}\geq a>b\geq\mathbf{1}$.
The faces are denoted $(abcd)$ with $\mathbf{8}\geq a>b>c>d\geq\mathbf{1}$.
Then
$$(\Fr{\kspan{(\mathbf{21}),\cdots,(\mathbf{4321}),\ldots,H}},\dd) \quism (\oC,0)$$
is generated by all edges and faces and $H$ so that the edges are of degree $0$ and $\inp{(ab)}=b$, $\out{(ab)}=a$; faces are of degree $1$ and $\inp{(abcd)}=d$, $\out{(abcd)}=a$; finally $\dg{H}=2$ and $\inp{H}=\mathbf{1}$, $\out{H}=\mathbf{8}$.
The differential is given by
\begin{align*}
\dd(ab) &= 0,\\
\dd(abcd) &= (ac)(cd)-(ab)(bd),\\
\dd H &= (\mathbf{84})(\mathbf{4321}) + (\mathbf{8743})(\mathbf{31}) - (\mathbf{8642})(\mathbf{21})\ + \\
\phantom{\dd H} &\phantom{=} + (\mathbf{87})(\mathbf{7531}) - (\mathbf{8765})(\mathbf{51}) - (\mathbf{86})(\mathbf{6521}).
\end{align*}
The resolving morphism maps edges to edges and all other generators to $0$.
It is easy to verify that this is a minimal resolution.

We let the reader convince himself that $\oC$ can't indeed be resolved just in degrees $0$ and $1$.
Rigorously, this would follow from the uniqueness of the minimal resolution together with a theorem asserting that any free resolution decomposes into a free product of a minimal resolution and an acyclic dg \colop{V}\footnote{An analogue exists in rational homotopy theory - see \cite{FHT}, Theorem 14.9.}.
These theorems however go beyond the scope of this paper.
\end{example}

\begin{example} \label{ISO}
An explicit resolution of the category generated by
\begin{diagram}
V_1 & \pile{\rTo^f\\ \lTo_g} & V_2
\end{diagram}
with relations $$fg-\id_{V_2},\quad gf-\id_{V_1}$$
was found in \cite{IPL}.
It contains a generator of \emph{each} nonnegative degree.
\end{example}

%%%%%%%%%%%%%%%
%%%%%%%%%%%%%%%
\subsection{\texorpdfstring{$\chir{-}_n$ maps}{[[-]]\_n maps}} \label{SectionChiral}
%%%%%%%%%%%%%%%
%%%%%%%%%%%%%%%

Since $\oCres$ is concentrated in arity $1$, we won't distinguish between $\oCres$ and $\oCres(1)$.
Also observe, that \colop{V} concentrated in arity $1$ is just a coloured\footnote{The operations are defined only partially, respecting the colours.} dg associative algebra.

Consider the usual dg structure on $\otexp{\oCres}{n}$.
There is also a right action of $\Sigma_n$ generated by transpositions as follows. 
Let $\tau\in\Sigma_n$ exchange $i$ and $j$. 
Then 
\begin{gather*}
(r_1\ot\cdots\ot r_i\ot\cdots\ot r_j\ot\cdots\ot r_n)\cdot\tau := \\
= (-1)^{\dg{r_i}\dg{r_j}+(\dg{r_i}+\dg{r_j})\sum_{i<k<j}\dg{r_k}} r_1\ot\cdots\ot r_j\ot\cdots\ot r_i\ot\cdots\ot r_n
\end{gather*}
for any $r_1,\ldots,r_n\in\oCres$ such that $\inp{r_i}=\out{s_i}$ for all $1\leq i\leq n$.
Further, there is the factorwise composition on $\otexp{\oCres}{n}$:
\begin{gather} \label{EQFacwiseComp}
(r_1\ot\cdots\ot r_n)\circ(s_1\ot\cdots\ot s_n):=(-1)^{\sum_{n\geq i>j\geq 1} \dg{r_i}\dg{s_j}}(r_1s_1)\ot\cdots\ot(r_ns_n).
\end{gather}
It is easily seen that $\dd$ is a degree $-1$ derivation with respect to $\oo$:
$$\dd(R\circ S)=(\dd R)\circ S+(-1)^{\dg{R}}R\circ \dd S$$
for any $R,S\in\otexp{\oCres}{n}$.
Also, $\oo$ is $\Sigma_n$ equivariant: $$(R\oo S)\cdot\tau = (R\cdot\tau)\oo(S\cdot\tau).$$

The following lemma is a straightforward generalization of Definition $23$ of \cite{HDA}:

\begin{lemma} \label{chiral}
For every integer $n\geq 1$, there is a linear map
$$\chir{-}_n:\oCres\to\otexp{\oCres}{n}$$
satisfying for every $r,r'\in\oCres$
\begin{enumerate}
\item[(C1)] $\chir{r}_n$ is $\Sigma_n$-stable,
\item[(C2)] $\chir{r}_n\in\otexp{\oCres\binom{\out{r}}{\inp{r}}}{n}$,
\item[(C3)]	$\deg{\chir{r}_n} = \deg{r}$,
\item[(C4)] $\chir{f}_n=\otexp f n$ for every morphism $f\in M\subset F_0$ (recall \ref{CatResCond}),
\item[(C5)] $\chir{r\oo r'}_n=\chir{r}_n\oo\chir{r'}_n$,
\item[(C6)] $\dd\chir{r}_n=\chir{\dd r}_n$.
\end{enumerate}
\end{lemma}

\begin{proof}
Fix $n$.
We proceed by induction on degree $d$.
(C4) defines $\chir{-}_n$ for $M$, we extend linearly to $F_0$ and then extend by (C5) to all of $\Fr{F_0}$.
Obviously, (C1)--(C6) hold for $r,r'\in\Fr{F_0}$.
Assume we have already defined $\chir{-}_n$ on $\Fr{F_{<d}}$ so that (C2)--(C6) hold.
\begin{enumerate}
\item Let $d=1$.
By the assumptions \ref{CatResCond}, $F_1=\kspan{R}$ and $f\in R$.
We have $\dd f=r_1-r_2$ as in \ref{CatResCond}, hence $\chir{\dd f}_n = \otexp{r_1}{n} - \otexp{r_2}{n}$.
Define 
\begin{gather} \label{AHOJAHOJ}
\chir{f}_n^{\mathrm{NS}} := \sum_{i=0}^{n-1} \otexp{r_1}{i}\ot f\ot\otexp{r_2}{n-i-1}.
\end{gather}
An easy computation shows $\chir{f}_n^{\mathrm{NS}}$ is a degree $1$ element of $\otexp{\oCres}{n}\binom{\out{f}}{\inp{f}}$ satisfying $\dd\chir{f}_n^{\mathrm{NS}} = \chir{\dd f}_n$.
\item Let $d\geq 2$.
Since $\dg{\dd f}<d$, $\chir{\dd f}_n$ is already constructed and we are solving the equation $$\dd\chir{f}_n=\chir{\dd f}_n$$ for an unknown $\chir{f}_n$ in the standard way.
By the induction assumption, $\dd\chir{\dd f}_n = \chir{\dd^2 f}_n = 0$.
By the usual K\"unneth formula, $\otexp{\oCres}{n}$ is acyclic in positive degrees.
Since $\dg{\chir{\dd f}_n}=d-1>0$, we obtain a degree $d$ element $\chir{f}_n^{\mathrm{NS}} \in \otexp{\oCres\binom{\out{f}}{\inp{f}}}{n}$ such that $\dd\chir{f}_n^{\mathrm{NS}} = \chir{\dd f}_n$.
\end{enumerate}
Making $\chir{f}_n^{\mathrm{NS}}$ to satisfy (C1) in characteristics $0$ is easy:
\begin{gather} \label{EQSigmaBalance}
\chir{f}_n := \frac{1}{n!}\sum_{\sigma\in\Sigma_n} \chir{f}_n^{\mathrm{NS}}\cdot\sigma.
\end{gather}
We now have $\chir{f}_n$ satisfying (C1)--(C4) and (C6) for every $f\in F_d$.
Extend this to $\Fr{F_{\leq d}}$ by (C5).
By $\Sigma_n$ equivariance of $\oo$, (C1) holds on $\Fr{F_{\leq d}}$.
Verifying (C2) and (C3) is trivial, hence it remains to check (C6).
Let $f_1,\ldots,f_m\in F_{\leq d}$:
\begin{align*}
\dd\chir{f_1\cdots f_m}_n &= \dd\left( \chir{f_1}_n\oo\cdots\oo\chir{f_m}_n \right) = \sum_{i=1}^m (-1)^{\epsilon_i} \chir{f_1}_n\oo\cdots\oo\dd\chir{f_i}_n\oo\cdots\oo\chir{f_m}_n = \\
&= \sum_{i=1}^m (-1)^{\epsilon_i} \chir{f_1}_n\oo\cdots\oo\chir{\dd f_i}_n\oo\cdots\oo\chir{f_m}_n = \\
&= \chir{\sum_{i=1}^m(-1)^{\epsilon_i} f_1\cdots\dd f_i\cdots f_m}_n = \chir{\dd(f_1\cdots f_m)}_n,
\end{align*}
where $\epsilon_i := {\dg{f_1}+\cdots +\dg{f_{i-1}}}$.
Hence (C6) is valid for all elements of $\Fr{F_{\leq d}}$ and the induction is finished.
\end{proof}

\begin{example}
If $\oCres$ is concentrated in degrees $\leq 1$, then we have explicit formulas \eqref{AHOJAHOJ} and \eqref{EQSigmaBalance} for $\chir{-}$ given in the proof.
\end{example}

\begin{example}
For the resolution of Example \ref{ISO}, the construction of $\chir{-}_n$'s using Lemma \ref{chiral} is not explicit.
In this case, $\chir{-}$ was found explicitly in \cite{HDA}, Remark 25.
\end{example}

The following lemma shows that $\chir{-}_2$ induces $\chir{-}_n$ for all $n\geq 3$.
$\chir{-}_2$ can be thought of as a coproduct on $\oCres$.
If $\chir{-}_2$ is moreover coassociative, then (C2),(C5) and (C6) means that $(\oCres,\oo,\chir{-}_2)$ is a coloured dg \emph{bialgebra}.

\begin{lemma} \label{EasyLemma}
Let $\chir{-}^\mathrm{NS}_2:\oCres\to\oCres\ot\oCres$ be a linear map satisfying the conditions (C2)--(C6) of Lemma \ref{chiral}.
Set 
\begin{gather} \label{EQDefChir}
\chir{-}^\mathrm{NS}_n := (\chir{-}^\mathrm{NS}_2\ot\id[n-2]) (\chir{-}^\mathrm{NS}_2\ot\id[n-3]) \cdots (\chir{-}^\mathrm{NS}_2\ot\id) \chir{-}^\mathrm{NS}_2.
\end{gather}
Then for each $n\geq 3$, $\chir{-}^\mathrm{NS}_n$ satisfies (C2)--(C6) and $\chir{-}_n$ defined by \eqref{EQSigmaBalance} satisfies (C1)--(C6).
If $\chir{-}^\mathrm{NS}_2$ is coassociative, i.e. $(\chir{-}^\mathrm{NS}_2\ot\id)\chir{-}^\mathrm{NS}_2 = (\id\ot\chir{-}^\mathrm{NS}_2)\chir{-}^\mathrm{NS}_2$, then
$$(\id[i]\ot\chir{-}^\mathrm{NS}_a\ot\id[b-i-1]) \chir{-}^\mathrm{NS}_b = \chir{-}^\mathrm{NS}_{a+b-1}$$
for every $a,b\geq 2$, $0\leq i\leq b-1$.
\end{lemma}

\begin{proof}
Conditions (C2)--(C4) for $\chir{-}^\mathrm{NS}_n$ are easily seen to be satisfied.

We sketch a proof of (C5) by the standard flow diagrams.
Let $\chir{-}^\mathrm{NS}_2$ be represented by $\PICcm$, then $\chir{-}^\mathrm{NS}_n$ is represented by $\overbrace{\PICcmMore}^n$.
Let $\oo:\oCres\ot\oCres\to\oCres$ of \eqref{EQFacwiseComp} be represented by $\PICm$, then $\oo:\otexp{\oCres}{n}\ot\otexp{\oCres}{n}\to\otexp{\oCres}{n}$ is represented, e.g. for $n=3$, by $\PICfmThree$.
Observe that the signs are handled by the Koszul sign convention.
The property (C5) for $\chir{-}^\mathrm{NS}_2$ states 
\begin{gather} \label{EQBialg}
\PICiks = \PICcross.
\end{gather}
For $n=3$, we have to prove $\chir{a\oo b}^\mathrm{NS}_3 = \chir{a}^\mathrm{NS}_3 \oo \chir{b}^\mathrm{NS}_3$, i.e. $$\PICiksThree = \PICcrossThreebold.$$
Applying \eqref{EQBialg} to the bold subgraph, we obtain $$\PICinterm$$
and another application of \eqref{EQBialg} on the bold subgraph gives the left hand side of the desired equality.
The general case is analogous.

We prove (C6) by induction on $n$.
$n=2$ is the hypothesis.
Let (C6) be true for $n-1$ and let's compute:
\begin{align*}
\dd\chir{-}^\mathrm{NS}_n &= \dd(\chir{-}^\mathrm{NS}_2\ot\id[n-2]) \cdots (\chir{-}^\mathrm{NS}_2\ot\id)\chir{-}^\mathrm{NS}_2 = \\
&= (\dd\chir{-}^\mathrm{NS}_2\ot\id[n-2])\chir{-}^\mathrm{NS}_{n-1} + \sum_{i=0}^{n-3} (\chir{-}^\mathrm{NS}_2\ot\id[i]\ot\dd\ot\id[n-3-i])\chir{-}^\mathrm{NS}_{n-1} = \\
&= (\chir{-}^\mathrm{NS}_2\ot\id[n-2])(\dd\ot\id[n-2])\chir{-}^\mathrm{NS}_{n-1}\ + \\
&\phantom{=} + \sum_{i=0}^{n-3} (\chir{-}^\mathrm{NS}_2\ot\id[n-2]) (\id[i+1]\ot\dd\ot\id[n-3-i])\chir{-}^\mathrm{NS}_{n-1} = \\
&= (\chir{-}^\mathrm{NS}_2\ot\id[n-2])\dd\chir{-}^\mathrm{NS}_{n-1} = (\chir{-}^\mathrm{NS}_2\ot\id[n-2])\chir{-}^\mathrm{NS}_{n-1}\dd = \chir{-}^\mathrm{NS}_{n}\dd.
\end{align*}

The proof of the coassociativity statement is easy and we leave it to the reader.
\end{proof}

Later, in Theorem \ref{BarCobarChiral}, we will construct $\chir{-}$ on the bar-cobar resolution $\BC{\oC}$ of $\oC$ using this lemma out of a coassociative coproduct on $\BC{\oC}$.

\begin{example}
Our assumptions \ref{CatResCond} are important.
Consider the category generated by a single morphism between two distinct objects as in Example \ref{SingleMorphism}.
Then $\oC$ has yet another resolution:
Take the same generators as in \ref{SingleMorphism},
$$\oCres := (\Fr{\kspan{f,g,H}},\dd) \quism (\oC,0),$$
but let $$\dd H = f + g.$$
Then an elementary linear algebra shows that $\otexp{f}{2}+\otexp{g}{2}$ is a cycle but not a boundary in $\otexp{\oCres}{2}$.
Hence our proof of Lemma \ref{chiral} would fail.
\end{example}

%%%%%%%%%%%%%%%%
%%%%%%%%%%%%%%%%
\subsection{Operad describing diagrams} \label{SectionOperad}
%%%%%%%%%%%%%%%%
%%%%%%%%%%%%%%%%

Let a small category $\ucat$ (together with its operadic version \eqref{EQOpeVerCat}) and a dg operad $\oA$ be given.
A ($\ucat$-shaped) \textbf{diagram} of ($\oA$-algebras) is a functor $$D:\ucat\to\oA\mathsf{-algebras}.$$
Now we describe a dg \colop{V} $\oD$ such that $\oD$-algebras are precisely $\ucat$-shaped diagrams.
We denote by $*$ the free product of dg \colop{V}, i.e. the coproduct in the category of dg \colop{V}.

\begin{definition} \label{EQDefOfDiagramFunctor}
For any (noncoloured) dg operad $(\oA,\dd_\oA)$, define
\begin{gather} \label{oDdefinition}
(\oA,\dd_{\oA})_{\ucat} := \left( \frac{( \bigfpr_{v\in V} \oA_v)\fpr\oC}{(fa_{\inp{f}}-a_{\out{f}}\otexp{f}{\ar{a}}\ |\ a\in\oA,\ f\in\Mor\ucat)},  \dd\right),
\end{gather}
where $\oA_v$ is a copy of $\oA$ concentrated in colour $v$ and symbols for its elements are decorated with lower index $v$.
Let the differential $\dd$ be defined by formulas
\begin{align*}
&\dd a_v = (\dd_{\oA} a)_v,\\
&\dd f=0
\end{align*}
for any  $a\in\oA$, $v\in V$ and $f\in C$.
For a dg operad morphism $(\oA,\dd_\oA)\To{\xi}(\oB,\dd_\oB)$, a dg \colop{V} morphism $$(\oA,\dd_\oA)_{\ucat} \To{\xi_{\ucat}} (\oB,\dd_\oB)_{\ucat}$$ is defined by
\begin{align}
&\xi_\ucat(a_v) := (\xi(a))_v,\label{UcatFtorOnMorph}\\
&\xi_\ucat(f) := f. \nonumber
\end{align}
It is easy to verify that the defining ideal of $\oA_\ucat$ is sent to the defining ideal of $\oB_\ucat$ and also that $\xi_\ucat\dd=\dd\xi_\ucat$, thus $\xi_\ucat$ is well defined.
It is also easily seen that $$\xi_\ucat \zeta_\ucat = (\xi\zeta)_\ucat$$ for any two dg operad morphisms $\xi,\zeta$, hence $$-_\ucat:\textsf{dg operads}\to\mathsf{dg}\ V\mathsf{\!-operads}$$ is a functor.
\end{definition}

Set $$\oD:=(\oA,\dd_\oA)_\ucat.$$
It is immediately seen that the functor $D$ above is essentially the same thing as $\oD$-algebra, i.e. dg \colop{V} morphism $\oD\to\oEnd_W$, where $W=\bigoplus_{v\in V}D(v)$ and each $D(v)$ is a dg $\fld$-module of colour $v$.

The following lemma generalizes Proposition $5$ of \cite{HDA}.

\begin{lemma} \label{PhiQuism}
$-_{\ucat}$ preserves quisms.
\end{lemma}

\begin{proof}
Let $\oA=(\oA,\dd)$ be a dg \colop{V} and let $v,v_1,\ldots,v_n\in V$.
We claim that there is an isomorphism 
\begin{align*}
\oA_\ucat\binom{v}{v_1,\ldots,v_n} &\cong\,\oA_v(n)\ot_{\Sigma_n}\left( \bigoplus_{\vec{W}} \oC\binom{v}{w_{1,1}}\ot\cdots\ot \oC\binom{v}{w_{n,1}}\ot\kspan{\Sigma(\vec{W},\vec{v})} \right) = \\
&= (\oA_v\oo\oC)\binom{v}{v_1,\ldots,v_n}
\end{align*}
of dg $\kspan{\Sigma_{\vec{v}}}$-modules, where $\vec{W}=(w_{1,1},\ldots,w_{n,1})\in V^n$.

The isomorphism assigns a \emph{canonical form} to an element $x\in \oA_\ucat\binom{v}{v_1,\ldots,v_n}$ :
Assume $x$ is an equivalence class of a composition of the generators from $(\bigfpr_{v\in V} \oA_v)\fpr\oC$.
Now use the defining relations to \uv{move} the generators from $\bigfpr_{v\in V}A_v$ to the left, so that $x=a\ot f_1\ot\cdots\ot f_n\ot \sigma$ for some $a\in \oA$, $f_1,\ldots,f_n\in \oC$ and $\sigma\in\Sigma_n$.
Then $a\ot f_1\ot\cdots\ot f_n\ot\sigma$ is called the canonical form of $x$.
By freeness, it is uniquely determined by $x$.
It is immediate that we get an isomorphism of dg $\kspan{\Sigma_{\vec{v}}}$-modules above and also $\oA_\ucat(n)\cong(\oA_v\oo\oC)(n)$ as dg $\kspan{\Sigma_n}$-modules.

Let $(\oA,\dd_\oA)\To{\xi}(\oB,\dd_\oB)$ be a quism.
It is easy to see that the following diagram commutes
\begin{diagram}
\oA_\ucat\binom{v}{v_1,\ldots,v_n} & \rTo^{\xi_\ucat} & \oB_\ucat\binom{v}{v_1,\ldots,v_n} \\
\dTo<\cong & & \dTo>{\cong} \\
(\oA_v\oo\oC)\binom{v}{v_1,\ldots,v_n} & \rTo^{\xi\ot\id\ot\cdots\ot\id\ot\id} & (\oB_v\oo\oC)\binom{v}{v_1,\ldots,v_n}
\end{diagram}
The diagram descends to homology, the lower horizontal arrow becomes an isomorphism by Lemma \ref{YetAnotherKunneth}, thus the upper horizontal arrow becomes an isomorphism as well.
\end{proof}

%%%%%%%%%%%%%%%%
%%%%%%%%%%%%%%%%
\subsection{Main theorem} \label{SectionTheorem}
%%%%%%%%%%%%%%%%
%%%%%%%%%%%%%%%%

Suppose we are given a resolution
$$\oCres = (\Fr{F},\dd) \quism[\phi_\oC] (\oC,0)$$
of $\oC$ satisfying the \emph{assumptions} \ref{CatResCond}
and a \emph{minimal} resolution
$$\oAres = (\Fr{X},\dd) \quism[\phi_{\oA}] (\oA,\dd)$$
of $\oA$.
We will use the same symbol $\dd$ for all the involved differentials.
The correct meaning will always be clear from the context.
Denote
\begin{align*}
X_V &:= X\ot\kspan{V} \\
X_F &:= \upar X\ot F.
\end{align*}
These are \smod[V]s by $\Sigma$ action on the $X$ factor.
An element $x\ot v\in X\ot\kspan{V}$ is denoted by $x_v$.
Analogously, $\upar x\ot f \in \upar X\ot F$ is denoted $x_f$.
Hence $$\dg{x_v}=\dg{x},\quad \dg{x_f}=\dg{x}+\dg{f}+1.$$
Obviously $X_V = \bigoplus_{v\in V}X\ot\kspan{v}$ and for any $v\in V$ we denote $$X_v:=X\ot\kspan{v}.$$
Finally, let
$$\oDres := \Fr{X_V\op F\op X_F}.$$
We also extend the notation $x_v$ for $x\in X$ and $v\in\kspan{V}$ to an operad morphism
\begin{align*}
-_v: \Fr{X} &\to \Fr{X_v} \into \oDres \\
x &\mapsto x_v.
\end{align*}

We will be interested in differentials of a special form on $\oDres$.
To state it precisely, we introduce the following maps:
\begin{definition} \label{DefPolarization}
For any $x\in X(n)$, the linear map $$\pol{x}{-}:\oCres\to\oDres(n)$$ is uniquely given by requiring
\begin{align*}
&\pol{x}{f} = x_f,\\
&\pol{x}{r_1r_2} = \pol{x}{r_1}\chir{r_2}_n + (-1)^{\dg{r_1}(\dg{x}+1)}r_1\pol{x}{r_2}
\end{align*}
for every $f\in F$ and $r_1,r_2\in \oCres$.

Thus $\pol{x}{-}$ behaves much like a derivation of degree $\dg{x}+1$.
Checking it is well defined boils down to verify $\pol{x}{r_1(r_2r_3)}=\pol{x}{(r_1r_2)r_3}$, which is easy.
Note that $\pol{x}{\id}=0$ for any unit in the \colop{V} $\oCres$.
\end{definition}

\begin{definition}
Let $\oA$ be a graded operad.
Recall that a presentation 
\begin{gather} \label{Presentation}
\frac{\Fr{E}}{(R)} \cong \oA
\end{gather}
is called \emph{quadratic} iff $R\subset \Fr[2]{E}$, i.e. elements of $R$ are sums of operadic compositions of exactly $2$ generators from $E$.
The elements of the \smod $E$ are called \emph{generating operations}.

Recall $\oA$ is called \emph{Koszul} iff there is a quadratic presentation \eqref{Presentation} such that the cobar construction on the Koszul dual $\Kdual{\oA}$ of $\oA$ is a resolution of $\oA$, i.e. $$\Omega(\Kdual{\oA}) \quism[\phi_\oA] (\oA,0).$$
See \cite{LV} for the notation and more details.
\end{definition}

We are now finally able to state our main result:

\begin{theorem} \label{MAIN}
Let $\oA$ be a \emph{Koszul} operad with generating operations concentrated in a \emph{single arity} $\geq 2$ and a \emph{single degree} $\geq 0$.
Let $\ucat$ be a small category and let $(\oCres,\dd_\oC) \To{\phi_{\oC}} (\oC,0)$ be its resolution (in the sense explained in Section \ref{SectionCategory}) satisfying the assumptions \ref{CatResCond}.
Then the graded \colop{V} $\oD=(\oA,0)_\ucat$ of \eqref{oDdefinition}, describing $\ucat$-shaped diagrams of $\oA$-algebras, has a free resolution $$(\oDres,\dd) \quism[\Phi] (\oD,0)$$ of the form $$\oDres := \Fr{X_V\op F\op X_F}$$
with the differential $\dd$ given by
\begin{align}
&\dd x_v = (\dd x)_v, \nonumber \\
&\dd f = \dd_{\oC}f, \label{DiffOnDInfty}  \\
&\dd x_f = (-1)^{1+\dg{x}}\pol{x}{\dd f} + (-1)^{1+\dg{x}\dg{f}}fx_{\inp{f}} + x_{\out{f}}\chir{f}_n + \omega(x,f), \nonumber
\end{align}
where $x\in X(n)$, $v\in V$, $f\in F$ and $\omega(x,f)$ lies in the arity $n$ part of the 
\begin{gather} \label{EQIdeal}
\textrm{ideal }\ideal{n}\textrm{ generated by }F_{\geq 1}\op X_F(<n)
\end{gather}
in
\begin{gather} \label{EQoDresLess}
\oDres^{<n}:=\Fr{X_V(<n)\op F\op X_F(<n)}.
\end{gather}
The differential $\dd$ on $\oDres$ is minimal iff $\dd$ on $\oCres$ is.
The dg \colop{V} morphism $\Phi$ is given by
\begin{align*}
\Phi(x_v) &= (\phi_{\oA}(x))_v,\\
\Phi(f) &= \phi_{\oC}(f),\\
\Phi(x_f) &= 0.
\end{align*}
\end{theorem}

\begin{remark}
This is a weaker form of Conjecture $31$ of \cite{HDA}.
First, we are restricted to Koszul operad $\oA$ with generating operations in a single arity and a single degree, while the conjecture lets $\oA$ be any dg operad.
Second, the ideal $\ideal{n}$ is larger, generated by $F_{\geq 1} \op X_F(<n)$, while the conjectured
\begin{gather} \label{EQOrigIdeal}
\textrm{ideal $\ideal{n}_{\textrm{orig}}$ is generated just by }X_F(<n).
\end{gather}
In particular, we recover, at least for $\oA$ as above, Theorem $7$ of \cite{HDA} dealing with the case of $\ucat$ being a single morphism between two distinct objects and $\oCres$ its trivial resolution.
Observe that in this case, $\ideal{n}$ is in fact generated just by $X_F(<n)$ since $F_{\geq 1}=0$ (of course, similar statement holds for any $\oCres$ concentrated in degree $0$, which corresponds to a \emph{free} category $\ucat$).
We also recover Theorems $18$ and $24$ of \cite{HDA}, again with the above mentioned restrictions.

However, there seems to be a completely unclear statement at the very end of the proof of Theorem $7$, page $11$ of \cite{HDA}.
As the proofs of Theorems $18$ and $24$ of \cite{HDA} are only sketched, there is probably the same problem.
To remedy it, we had to introduce our assumptions.
We will discuss these assumptions in detail after proving our main theorem.
However, we don't know any counterexample to the original theorems of \cite{HDA}.
\end{remark}

%%%%%%%%%%%%%%%%%%%%%%

\section{Proof of main theorem} \label{SectionProof}

\subsection{Lemmas} \label{SectionLemmas}

\begin{lemma} \label{lemmaA}
Let $\ucat$ be a small category, let $(\oCres,\dd_\oC) \To{\phi_{\oC}} (\oC,0)$ be its resolution satisfying the assumptions \ref{CatResCond}.
For \emph{any minimal} dg \colop{V} of the form $(\Fr{X},\dd)$ with $X(0)=X(1)=0$, let $$\oDres := \Fr{X_V\op F\op X_F}$$ and \emph{assume} there is a differential $\dd$ on $\oDres$ satisfying
\begin{align*}
&\dd x_v = (\dd x)_v, \\
&\dd f = \dd_\oC f,\\
&\dd x_f = (-1)^{1+\dg{x}}\pol{x}{\dd f} + (-1)^{1+\dg{x}\dg{f}}fx_{\inp{f}} + x_{\out{f}}\chir{f}_n + \omega(x,f),
\end{align*}
where $x\in X(n)$, $v\in V$, $f\in F$ and $$\omega(x,f) \in \oDres^{<n}=\Fr{X_V(<n)\op F\op X_F(<n)}(n).$$
\emph{Assume} $\phi$ is a dg \colop{V} morphism $$(\oDres,\dd )\xrightarrow{\phi}(\Fr{X},\dd)_{\ucat}$$ satisfying
\begin{align} \label{DefPhi}
&\phi(x_v) = x_v,\\
&\phi(f) = f \nonumber
\end{align}
for $f\in F_0$ and vanishing on all the other generators.
Then $\phi$ is a quism.
\end{lemma}

We use the symbol $x_v$ either for $x_v\in X_V\subset\oDres$ or $x_v\in\Fr{X}_v\subset (\Fr{X},\dd)_{\ucat}$.
Similarly for $f\in F_0$.
The correct meaning will always be clear from the context.

\begin{proof}
Let $\filt_i$ be the sub \smod[V] of $\oDres$ spanned by free compositions containing at least $-i$ generators from $X_V\op X_F$.
$\filt_i$'s form a filtration $$\cdots\subset\filt_{-2}\subset\filt_{-1}\subset\filt_0=\oDres.$$
$\pol{x}{\dd f}\in\filt_{-1}$ is obvious and $\ar{\omega(x,f)}=\ar{x}\geq 2$ implies $\omega(x,f)\in\filt_{-1}$.
Hence $\dd\filt_i\subset\filt_i$.
Since $X$ contains no elements of arity $0$ and $1$, for \emph{a fixed arity} $n$ the arity $n$ part $\filt_i(n)$ of this filtration is bounded below.
Consider the corresponding spectral sequence $(E^*(n),\dd^*(n))$.
For each $n$, $(E^*(n),\dd^*(n))$ converges by the classical convergence theorem.
We collect these spectral sequences into $(E^*,\dd^*)$.
Recall that each $(E^i,\dd^i)$ is a dg \colop{V}.
In the sequel, such arity-wise constructions will be understood without mentioning the arity explicitly.
For the $0^\textrm{th}$ term, we have $$E^0\cong\oDres$$
as graded \colop{V}.
Now we make $\dd^0$ explicit.
Let $x\in X(n)$, $n\geq 2$.
By the minimality, each summand of $\dd x_v$ contains at least $2$ generators from $X_V$, hence $\dd x_v\in\filt_{-2}$ and $\dd^0 x_v=0$.
Next, observe that for $n=2$, $\omega(x,f)=0$ by arity reasons.
Let $n\geq 3$.
Each summand of $\omega(x,f)\neq 0$ contains only generators of arity $<n$, hence at least $2$ of these are of arities $\geq 2$.
But generators of arity $\geq 2$ come from $X_V\op X_F$, i.e. $\omega(x,f)\in\filt_{-2}$.
Hence the differential $\dd^0$ is the derivation determined by formulas
\begin{align*}
&\dd^0 x_v = 0\\
&\dd^0 f = \dd f\\
&\dd^0 x_f = (-1)^{1+\dg{x}}\pol{x}{\dd f}+(-1)^{1+\dg{x}\dg{f}}fx_{\inp{f}}+x_{\out{f}}\chir{f}_n
\end{align*}
for $x\in X(n)$ and $f\in F$.

There is a similar construction on $(\Fr{X},\dd)_{\ucat}$.
Denote $\dd'$ its differential.
Let $\filt'_i$ be the sub \smod[V] of $(\Fr{X},\dd)_{\ucat}$ spanned by free compositions containing at least $-i$ generators from $X_V$.
Then these form a filtration $$\cdots\subset\filt'_{-2}\subset\filt'_{-1}\subset\filt'_0=(\Fr{X},\dd)_{\ucat}.$$
Obviously $\dd'\filt'_i\subset\filt'_i$.
By the same argument as above, this filtration is bounded below and hence the corresponding spectral sequence $(E'^*,\dd'^*)$ converges.
For the $0^\textrm{th}$ page, we have $$(E'^0,\dd'^0)\cong (\Fr{X},0)_{\ucat}$$
as dg \colop{V}, i.e. $$\dd'^0=0.$$

The dg \colop{V} morphism $\phi$ satisfies $\dd\filt_i\subset\filt'_i$, hence it induces a morphism $\phi^*:(E^*,\dd^*)\to(E'^*,\dd'^*)$ of spectral sequences.
By \cite{W}, Theorems 5.2.12 and 5.5.1, to prove that $\phi$ is quism, it suffices to show that $\phi^0$ is a quism.
We will prove
\begin{gather} \label{MainEQ}
H_*(E^0,\dd^0) = \frac{\Fr{X_V\op F_0}}{ \left( \set{-fx_{\inp{f}}+x_{\out{f}}\chir{f}_{\ar{x}}}{x\in X,\ f\in F_0} \cup \dd F_1 \right) },
\end{gather}
compare with \eqref{EQNote} and Definition \ref{EQDefOfDiagramFunctor}.
This implies $H_*(\phi^0)$ is the identity and we are done.

The dg \colop{V} $(E^0,\dd^0)$ carries a filtration $$0=\filt''_{-1}\subset\filt''_0\subset\filt''_1\subset\cdots,$$ where $\filt''_i$ is sub \smod[V] of $E^0$ spanned by compositions with $$(\textrm{degree}+\textrm{number of generators from }X_V)\leq i.$$
Obviously $\dd^0\filt''_i\subset\filt''_i$.
This filtration is bounded below and exhaustive, hence the corresponding spectral sequence $(E^{0*},\dd^{0*})$ converges by \cite{W}, Theorem 5.2.12.
We have $$E^{00}\cong\oDres$$ as graded \colop{V} and
\begin{align*}
&\dd^{00}x_v = 0,\\
&\dd^{00}f = 0,\\
&\dd^{00}x_f = (-1)^{1+\dg{x}\dg{f}}fx_{\inp{f}}+x_{\out{f}}\chir{f}_n
\end{align*}
for $x\in X(n)$ and $f\in F$.
We will show 
\begin{gather} \label{subMainEQ}
H_*(E^{00},\dd^{00})=\frac{\Fr{X_V\op F}}{((-1)^{1+\dg{x}\dg{f}}fx_{\inp{f}}+x_{\out{f}}\chir{f}_{\ar{x}}\ |\ x\in X,\ f\in F)}.
\end{gather}
Assume this is already done and let's prove \eqref{MainEQ}.
We proceed to the $1^{\textrm{st}}$ page $E^{01}$ of $E^{0*}$:
$$E^{01}\cong H_*(E^{00},\dd^{00})$$
and under this isomorphism, $\dd^{01}$ is given by
\begin{align*}
&\dd^{01}x_v = 0,\\
&\dd^{01}f = \dd f.
\end{align*}
By the same argument as in the proof of Lemma \ref{PhiQuism}, $$(E^{01},\dd^{01})\cong (\Fr{X_V},0)\oo (\Fr{F},\dd).$$
By Lemma \ref{YetAnotherKunneth} and \eqref{EQNote}, $$H_*(E^{01},\dd^{01}) \cong \Fr{X_V}\oo\oC \cong \Fr{X_V}\oo\frac{\Fr{F_0}}{(\dd F_1)}$$ and by the argument of Lemma \ref{PhiQuism} again,
\begin{gather} \label{EqAux}
H_*(E^{01},\dd^{01}) \cong \frac{\Fr{X_V\op F_0}}{\left( \set{-fx_{\inp{f}}+x_{\out{f}}\chir{f}_{\ar{x}}}{x\in X,\ f\in F_0} \cup \dd F_1 \right) }.
\end{gather}
This is the $2^{\textrm{nd}}$ page $E^{02}$ and we claim that all the higher differentials vanish: $\dd^{0k}=0$ for $k\geq 2$.
To see this, let's assign \emph{inner degree}, denoted by $\indg{-}$, to generators of $E^{0}$:
$$\indg{x_v}=0,\quad \indg{f}=\dg{f},\quad \indg{x_f}=\dg{f}+1.$$
This extends to $E^0$ by requiring the operadic composition to be of inner degree $0$.
Now notice that $\dd^0$ is of inner degree $-1$ and so are all the differentials $\dd^{0k}$.
But \eqref{EqAux} is concentrated in inner degree $0$, hence the spectral sequence $(E^{0*},\dd^{0*})$ collapses as claimed.
We conclude that $E^{02}=E^{0\infty}\cong H_*(E^0,\dd^0)$, thus proving \eqref{MainEQ}.

It remains to prove \eqref{subMainEQ}.
Let $\filt'''_i$ be the sub \smod[V] of $E^{00}\cong\oDres$ spanned by compositions with at least $-i$ generators from $F\op X_F$.
Then $$\cdots\subset\filt'''_{-2}\subset\filt'''_{-1}\subset\filt'''_0=E^{00}$$ is a filtration with $\dd^{00}\filt'''_i\subset\filt'''_i$.
Denote $(E^{00*},\dd^{00*})$ the corresponding spectral sequence.
The convergence of this spectral sequence will be discussed later.
We have 
\begin{align*}
&E^{000}\cong\oDres,\\
&\dd^{000}x_v=0,\\
&\dd^{000}f=0,\\
&\dd^{000}x_f=(-1)^{1+\dg{x}\dg{f}}fx_{\inp{f}}.
\end{align*}
Now we prove 
\begin{gather} \label{EQKolikataUz}
H_*(E^{000},\dd^{000})=\frac{\Fr{X_V\op F}}{((-1)^{1+\dg{x}\dg{f}}fx_{\inp{f}}\ |\ x\in X,\ f\in F)}.
\end{gather}
First observe that $\Fr{F}\oo(X_V\op X_F)$ is closed under $\dd^{000}$ and 
\begin{gather} \label{EQHomoloAux}
H_*(\Fr{F}\oo(X_V\op X_F),\dd^{000})=X_V.
\end{gather}
In the sequel, we may drop the differentials from the notation \uv{$H_*(-,\dd)$} if no confusion can arise.
Let 
\begin{align*}
&P_0 := \kspan{\id}, \\
&P_{n+1} := \Fr{F} \opw \Fr{F}\oo(X_V\op X_F)\oo P_n
\end{align*}
for $n\geq 0$.
We immediately see that $P_n$'s are closed under $\dd^{000}$ and 
\begin{gather} \label{EQRecursionExpanded}
P_n=\bigoplus_{i=0}^{n-1}(\Fr{F}\oo(X_V\op X_F))^{\oo i}\oo\Fr{F} \opw (\Fr{F}\oo(X_V\op X_F))^{\oo n},
\end{gather}
where we used the (iterated) composition product \eqref{EQCompPoductExplicitly}.
By Lemma \ref{YetAnotherKunneth} and \eqref{EQHomoloAux},
$$H_*(P_n) \cong \bigoplus_{i=0}^{n-1}(X_V)^{\oo i}\oo \Fr{F} \opw (X_V)^{\oo n}.$$
\eqref{EQRecursionExpanded} provides a chain of inclusions
$$P_0 \into P_1 \into \cdots \to \colim_{\xrightarrow[n]{}}P_n \cong E^{000}$$
with direct limit $E^{000}$, as easily seen.
Since direct limits commute with homology,
\begin{align*}
H_*(E^{000},\dd) &\cong \colim_{\xrightarrow[i]{}} H_*(P_n) = \colim_{\xrightarrow[i]{}} \left(\bigoplus_{i=0}^{n-1}(X_V)^{\oo i}\oo \Fr{F} \opw (X_V)^{\oo n} \right) \cong\\
&\cong \Fr{X_V}\oo \Fr{F} \cong \frac{\Fr{X_V\op F}}{((-1)^{1+\dg{x}\dg{f}}fx_{\inp{f}}\ |\ x\in X,\ f\in F)}.
\end{align*}

The $1^\textrm{st}$ page $E^{001}$ is therefore described by \eqref{EQKolikataUz}.
An argument with inner degree analogous to the one above shows that $\dd^{00k}=0$ for $k\geq 1$:
In $(E^{00},\dd^{00})$, set $$\indg{x_v}=\indg{f}=0,\quad \indg{x_f}=1.$$
Hence $E^{001}=E^{00\infty}$ is the stable term.

Although we don't know how to prove the convergence of the spectral sequence $(E^{00*},\dd^{00*})$ directly (the filtration is bounded above but not below, we only have the Hausdorff property $\cap_{i}\filt'''_i=0$), there is a weaker statement which follows from Lemma 5.5.7 of \cite{W}:
The $i^{\textrm{th}}$ graded part $\filt'''_i H_*(E^{00},\dd^{00}) / \filt'''_{i-1} H_*(E^{00},\dd^{00})$ of the filtration on homology\footnote{Recall the usual notation $\filt'''_i H_*(E^{00},\dd^{00}) := \Im(H_*(\filt'''_i,\dd^{00})\to H_*(E^{00},\dd^{00}))$.} is isomorphic to a subspace $e_i$ of $E_i^{00\infty}$.\footnote{The lower index denotes the grading associated to the filtration $\filt'''_i$.}
We have (simplifying the notation)
\begin{gather*}
\frac{\Fr{X_V\op F}}{((-1)^{1+\dg{x}\dg{f}}fx_{\inp{f}}+x_{\out{f}}\chir{f}_{\ar{x}})} \subset H_*(E^{00}) \cong \bigoplus_{i\leq 0}\frac{\filt'''_iH_*(E^{00})}{\filt'''_{i-1}H_*(E^{00})} \cong \bigoplus_{i\leq 0}e_i \subset \bigoplus_{i\leq 0}E_p^\infty \cong\\
\cong \frac{\Fr{X_V\op F}}{((-1)^{1+\dg{x}\dg{f}}fx_{\inp{f}})},
\end{gather*}
where the first inclusion is the obvious part of \eqref{subMainEQ} and the second inclusion has just been discussed.
It is not difficult to map the left-hand side through all the isomorphisms and to see that it is mapped \emph{onto} the right-hand side.
Hence the first inclusion is in fact equality and we are done proving \eqref{subMainEQ} and consequently the whole Lemma \ref{lemmaA}.
\end{proof}

\begin{lemma} \label{lemmaB}
Let $\omega(x,f)$ of Lemma \ref{lemmaA} moreover satisfies $\omega(x,f)\in\ideal{n}$ (recall \eqref{EQIdeal}).
Then $\phi$, uniquely determined by \eqref{DefPhi} as a graded \colop{V} morphism, is automatically a dg \colop{V} morphism (i.e. $\phi$ commutes with the differentials).
\end{lemma}

\begin{proof}
We have to verify $\phi\dd =\dd\phi$ for generators from $X_V\op F\op X_F$.
The only nontrivial case concerns $X_F$: we have to verify $\phi\dd x_f=0$.
We have 
\begin{align*}
\phi\dd x_f &= (-1)^{1+\dg{x}}\phi(\pol{x}{\dd f})\ +\\ &\phantom{=}+(-1)^{1+\dg{x}\dg{f}}\phi(f)\phi(x_{\inp{f}}) + \phi(x_{\out{f}})\phi(\chir{f}) + \phi(\omega(x,f)).
\end{align*}
If $\dd f=0$, then the first term vanishes trivially.
If $\dd f\neq 0$, then each summand of $\pol{x}{\dd f}$ contains a generator from $X_F$ and hence the first term vanishes too.

By the definition of $\ideal{n}$, we have $\phi(\omega(x,f))=0$.

Hence it remains to prove $$(-1)^{1+\dg{x}\dg{f}}\phi(f)\phi(x_{\inp{f}}) + \phi(x_{\out{f}})\phi(\chir{f}) = 0.$$
If $\dg{f}>0$, we have $\phi f=0$ by definition.
Also $\dg{\chir{f}}>0$ and hence each summand of $\chir{f}$ contains a generator from $F_{\geq 1}$ and consequently $\phi(\chir{f})=0$.
For $\dg{f}=0$, we may assume $f\in M$ and we want to prove $-f x_{\inp{f}} + x_{\out{f}}\otexp{f}{n} = 0$ in $(\Fr{X},\dd)_{\ucat}$.
But this is exactly one of the defining relations of $(\Fr{X},\dd)_{\ucat}$.
\end{proof}

\begin{lemma} \label{KoszulGenerators}
Let an operad $\oA$ be Koszul with generating operations concentrated in a single arity $N\geq 2$ and a single degree $D\geq 0$.
Then for every generator $x$ of the minimal resolution of $\oA$ there is $k\geq 1$ such that $$\ar{x}=a_k:=1+(N-1)k,\quad \dg{x}=d_k:=-1+(D+1)k.$$
Moreover, there is $K$ (possibly $K=+\infty$) such that a generator of arity $a_k$ and degree $d_k$ exists iff $k<K$.
\end{lemma}

\begin{proof}
By Koszulity, we have the minimal resolution $$\Omega(\Kdual{\oA}) \quism (\oA,0)$$
given by the cobar construction $\Omega(\Kdual{\oA})=(\Fr{\downar\overline{\Kdual{\oA}}},\dd)$.
Assume $\oA$ has the quadratic presentation \eqref{Presentation}.
Recall that the Koszul dual $\Kdual{\oA}$ is the quadratic cooperad cogenerated by $\upar E$ with corelations $\upar^{2}E$, see \cite{LV}, 7.1.4.
Thus $\Kdual{\oA}$ is a sub \smod of $\Fr{\upar E}$, hence it is concentrated in arities $1+(N-1)k$ and degrees $(D+1)k$.
Hence $\downar\overline{\Kdual{\oA}}$ is concentrated in arities $a_k=1+(N-1)k$ and degrees $d_k=-1+(D+1)k$.

We give only a brief proof of the last claim of this lemma, since we won't need it in the sequel.
Suppose that for every $k<K$ a generator of arity $a_k$ and degree $d_k$ exists.
Further let there be no generator in arity $a_K$.
By the inductive construction of the minimal resolution, as described in the proof of Theorem $3.125$ of \cite{MSS}, the generators in the next possible arity $a_{K+1}$ have degree $\leq d_{K-1} + 2D + 1 = d_{K+1}-1$.
But the existence of any such generator would contradict the previous part of this lemma.
In the next arity, $a_{K+2}$, the generators would have to have degree $\leq d_{K-1} + 3D + 1 < d_{K+2}$.
And so on, hence there are no generators in arity $a_k$ for $k\geq K$.
We encourage the reader to go through the cases $D=0$ and $D=1$.
\end{proof}

\begin{lemma} \label{LemmaDiffExists}
Let an operad $\oA$ be Koszul with generating operations concentrated in a single arity $\geq 2$ and a single degree $\geq 0$.
Then for every $x\in X$ and $f\in F$, there is $\omega(x,f)\in\ideal{\ar{x}}$ as stated in Theorem \ref{MAIN}, i.e. the derivation $\dd$ defined by \eqref{DiffOnDInfty} is indeed a differential on $\oDres$.
\end{lemma}

To prove this lemma, it is convenient to extend $\omega(x,f)$'s to a linear map as follows.
Fix $x\in X(n)$.
The linear map
$$\omega(x,-):\oCres\to\oDres(n)$$
is uniquely determined by
\begin{align*}
&\textrm{(arbitrary) values }\omega(x,f)\textrm{ on }f\in F\textrm{ and}\\
&\omega(x,r_1r_2) = \omega(x,r_1)\chir{r_2}_n + (-1)^{\dg{r_1}\dg{x}} r_1\omega(x,r_2)
\end{align*}
for any $r_1,r_2\in\oCres$.

Thus $\omega(x,-)$ behaves much like a derivation of degree $\dg{x}$.
Checking it is well defined is similar to \ref{DefPolarization}.

\begin{lemma} \label{Formula}
For any $r\in\oCres$, the formula \eqref{DiffOnDInfty} with $r$ in place of $f$ still holds:
$$\dd\pol{x}{r} = (-1)^{1+\dg{x}}\pol{x}{\dd r} + (-1)^{1+\dg{r}\dg{x}}rx_{\inp{r}} + x_{\out{r}}\chir{r}_n + \omega(x,r).$$
\end{lemma}

The proof explains the $\pm$ signs in the definition \eqref{DiffOnDInfty} of $\dd$ on $\oDres$.

\begin{proof}
It suffices to prove the lemma for $r$ of the form $r=f_1f_2\cdots f_k$, where $f_i\in F$.
We proceed by induction on $k$.
The case $k=1$ is exactly formula \eqref{DiffOnDInfty}.
Let $k\geq 2$ and suppose the lemma holds for every sum of compositions of at most $k-1$ elements and let $r=r_1r_2$, where $r_1,r_2$ are compositions of at most $k-1$ generators from $F$.
Now we want to prove 
\begin{gather*}
\dd\pol{x}{r_1r_2} = \\
= (-1)^{1+\dg{x}}\pol{x}{\dd(r_1r_2)} + (-1)^{1+(\dg{r_1}+\dg{r_2})\dg{x}}r_1r_2x_{\inp{r_2}} + x_{\out{r_1}}\chir{r_1r_2}_n + \omega(x,r_1r_2).
\end{gather*}
It is a straightforward computation, we will compare Left-Hand Side and Right-Hand Side:
\begin{align*}
\textrm{LHS} &= \dd\left( \pol{x}{r_1}\chir{r_2}_n + (-1)^{\dg{r_1}(\dg{x}+1)} r_1\pol{x}{r_2} \right) =\\
&= \left( (-1)^{1+\dg{x}}\pol{x}{\dd r_1} + (-1)^{1+\dg{r_1}\dg{x}}r_1x_{\inp{r_1}} + x_{\out{r_1}}\chir{r_1}_n + \omega(x,r_1) \right) \chir{r_2}_n\ + \\
&\phantom{=} + (-1)^{\dg{x}+\dg{r_1}+1}\pol{x}{r_1}\chir{\dd r_2}_n + (-1)^{\dg{r_1}(\dg{x}+1)} (\dd r_1) \pol{x}{r_2}\ +\\
&\phantom{=} + (-1)^{\dg{r_1}(\dg{x}+1)+\dg{r_1}} r_1\left( (-1)^{1+\dg{x}}\pol{x}{\dd r_2} + (-1)^{1+\dg{r_2}\dg{x}}r_2x_{\inp{r_2}}\ +\right. \\ 
&\phantom{=} + x_{\out{r_2}}\chir{r_2}_n + \omega(x,r_2) \Big) \displaybreak[0] \\
\textrm{RHS} &= (-1)^{1+\dg{x}}\pol{x}{(\dd r_1)r_2)} + (-1)^{1+\dg{x}+\dg{r_1}}\pol{x}{r_1\dd r_2}\ + \\
&\phantom{=} + (-1)^{1+(\dg{r_1}+\dg{r_2})\dg{x}}r_1r_2x_{\inp{r_2}} + x_{\out{r_1}}\chir{r_1}_n\chir{r_2}_n\ +\\
&\phantom{=} + \omega(x,r_1)\chir{r_2}_n + (-1)^{\dg{x}\dg{r_1}}r_1\omega(x,r_2) =\\
&= (-1)^{1+\dg{x}}\pol{x}{\dd r_1}\chir{r_2}_n + (-1)^{1+\dg{x}+(\dg{x}+1)(\dg{r_1}-1)}(\dd r_1)\pol{x}{r_2}\ +\\
&\phantom{=} + (-1)^{1+\dg{x}+\dg{r_1}}\pol{x}{r_1}\chir{\dd r_2}_n + (-1)^{1+\dg{x}+\dg{r_1}+(\dg{x}+1)\dg{r_1}}r_1\pol{x}{\dd r_2}\ + \\
&\phantom{=} + (-1)^{1+(\dg{r_1}+\dg{r_2})\dg{x}}r_1r_2x_{\inp{r_2}} + x_{\out{r_1}}\chir{r_1}_n\chir{r_2}_n\ + \\
&\phantom{=} + \omega(x,r_1)\chir{r_2}_n + (-1)^{\dg{x}\dg{r_1}}r_1\omega(x,r_2)
\end{align*}
The proof is finished by a careful sign inspection.
\end{proof}

\begin{proof}[Proof of Lemma \ref{LemmaDiffExists}]
Let $x\in X(n)$ and $f\in F_d$.
First, we make a preliminary computation using the formula of Lemma \ref{Formula}:
\begin{align*}
\dd^2 x_f &= (-1)^{1+\dg{x}}\dd\pol{x}{\dd f} + (-1)^{1+\dg{x}\dg{f}}(\dd f)x_{\inp{f}} + (-1)^{1+\dg{x}\dg{f}+\dg{f}}f\dd x_{\inp{f}}\ + \\
&\phantom{=} + (\dd x_{\out{f}})\chir{f}_n + (-1)^{\dg{x}}x_{\out{f}}\chir{\dd f}_n + \dd\omega(x,f)=\cdots\\
&= (-1)^{1+\dg{x}}\omega(x,\dd f) + (-1)^{1+\dg{x}\dg{f}+\dg{f}}f\dd x_{\inp{f}}  + (\dd x_{\out{f}})\chir{f}_n + \dd\omega(x,f)
\end{align*}
The condition $\dd^2 x_f=0$ is equivalent to
$$\dd\omega(x,f) = (-1)^{\dg{x}}\omega(x,\dd f) + (-1)^{\dg{f}(\dg{x}+1)}f\dd x_{\inp{f}} - (\dd x_{\out{f}})\chir{f}_n =: \varphi(x,f).$$
To construct $\omega(x,f)$ so that $\dd^2 x_f=0$, we will inductively solve the equation 
\begin{gather} \label{Equation}
\dd\omega(x,f)=\varphi(x,f)
\end{gather}
for unknown $\omega(x,f)$.
We proceed by induction on arity $n$ of $x$ and simultaneously by induction on degree $d$ of $f$.

For $n=N$ (the arity of the generating operations of $\oA$) and  $d=0$, we have $\dd f=0=\dd x$, hence \eqref{Equation} becomes $\dd\omega(x,f)=0$, which has the trivial solution.

Fix $n$ and $d$.
Assume we have already constructed $\omega(x,f)\in\ideal{\ar{x}}$ for every $x\in X(<n)$ and $f$ of any degree and also for $x\in X(n)$ and $f\in F_{<d}$.
Let $x\in X(n)$ and $f\in F_d$.
Observe that $\varphi(x,f)\in\oDres^{<n}$ (recall \eqref{EQoDresLess}) by the induction assumption and minimality.
When we restrict $\phi:\oDres\to(\Fr{X},\dd)_{\ucat}$ to $\oDres^{<n}$, we get the graded \colop{V} morphism
$$\oDres^{<n} \To{\phi} (\Fr{X(<n)},\dd)_{\ucat}$$
denoted by the same symbol.
By the induction assumption, $\dd^2=0$ on $\oDres^{<n}$.
By Lemma~\ref{lemmaB}, $\phi$ is dg \colop{V} morphism.
By Lemma \ref{lemmaA}, $\phi$ is a quism.
In a moment, we will show
\begin{align}
\dd\varphi(x,f) &= 0, \label{Cond1}\\
\phi\varphi(x,f) &= 0. \label{Cond2}
\end{align}
This will imply the existence of $\omega(x,f)\in\oDres^{<n}$ such that $\dd\omega(x,f)=\varphi(x,f)$.
In fact, $\omega(x,f)\in\ideal{n}$.
To see this, assume a summand $S$ of $\omega(x,f)$ is a composition of generators none of which comes from $X_F$.
Hence $S$ is an operadic composition of $x_1,\cdots,x_a\in X_V(<n)$ and $f_1,\cdots,f_b\in F$.
By a degree count, we now show that at least one of $f_j$'s lies in $F_{\geq 1}$.
By Lemma \ref{KoszulGenerators}, let $x_i$ have arity $1+(N+1)k_i$ and degree $-1+(D+1)k_i$.
We have
$$
\dg{x}+\dg{f} = \dg{\omega(x,f)} = \dg{S} = \sum_{i=1}^a\dg{x_i} + \sum_{j=1}^b\dg{f_j} = \sum_i (-1+(D+1)k_i) + \sum_j \dg{f_j}
$$
hence
\begin{gather} \label{Aux1}
\sum_j \dg{f_j} = \dg{x} + \dg{f} + a - (D+1)\sum_i k_i.
\end{gather}
Now $$\ar{x} = \ar{S} = 1 + \sum_i(\ar{x_i}-1) = 1 + (N-1)\sum_i k_i$$
hence $$\dg{x}= -1 + (D+1)\sum_i k_i.$$
Substituting this into \eqref{Aux1}, we get $$\sum_j \dg{f_j} = \dg{f}+a-1.$$
We have the trivial estimate $\dg{f}\geq 0$.
Since $\ar{f_j}=1$ for any $j$ and $\ar{x_i}<n=\ar{S}$ for any $i$, we have $a\geq 2$.
Hence $$\sum_{j}\dg{f_j}\geq 1$$ and therefore one of $f_j$'s lies in $F_{\geq 1}$.

It remains to verify the conditions \eqref{Cond1}, \eqref{Cond2}.
For \eqref{Cond1}, we have
\begin{align*}
\dd\varphi(x,f) = (-1)^{\dg{x}}\dd\omega(x,\dd f) +(-1)^{\dg{f}(\dg{x}+1)}(\dd f)\dd x_{\inp{f}} + (-1)^{\dg{x}}(\dd x_{\out{f}})\chir{\dd f}_n.
\end{align*}
Lemma \ref{Formula} and the induction hypothesis imply $$\dd\omega(x,\dd f) = (-1)^{\dg{x}(\dg{f}-1)+\dg{f}-1}(\dd f)\dd x_{\inp{f}} - (\dd x_{\out{f}})\chir{\dd f}_n$$ and after substituting this into the previous equation, we get $\dd\varphi(x,f)=0$.

For \eqref{Cond2}, let $d=0$ first. Then $\varphi(x,f) = f\dd x_{\inp{f}} - (\dd x_{\out{f}})\chir{f}_n$, hence we have to verify
$$f\dd x_{\inp{f}} - (\dd x_{\out{f}})\otexp{f}{n}=0\quad\textrm{in }(\Fr{X(<n)},\dd)_{\ucat}.$$
This follows by the same argument as Lemma \ref{PhiQuism}.
Now let $d>0$.
By induction assumption, $\omega(x,\dd f)\in\ideal{n}$ and therefore $\phi\omega(x,\dd f)=0$.
Finally $\phi f=0=\phi\chir{f}_n$ by definition of $\phi$ since $\dg{f}=\dg{\chir{f}_n}=d>0$.
\end{proof}

Now we can finally prove the main theorem:

\begin{proof}[of Theorem \ref{MAIN}]
Decompose $\Phi$ into $$(\oDres,\dd) \To{\phi} (\Fr{X},\dd)_\ucat \To{(\phi_\oA)_\ucat} (\oA,0)_\ucat = (\oD,0).$$
The dg \colop{V} morphism $(\phi_\oA)_\ucat$ (recall Definition \ref{EQDefOfDiagramFunctor}) is a quism by Lemma \ref{PhiQuism}.
$\phi$ is the graded \colop{V} morphism of Lemma \ref{lemmaA}.
By Lemma \ref{LemmaDiffExists}, there are $\omega(x,f)$'s in $\ideal{\ar{x}}$ such that $\dd$ on $\oDres$ is indeed a differential.
By Lemma \ref{lemmaB}, $\phi$ is a dg \colop{V} morphism and finally, by Lemma \ref{lemmaA}, $\phi$ is a quism.
\end{proof}

%%%%%%%%%%%%%%%%%%%%%%
%%%%%%%%%%%%%%%%%%%%%%
\subsection{Discussion} \label{SectionDiscussion}
%%%%%%%%%%%%%%%%%%%%%%
%%%%%%%%%%%%%%%%%%%%%%

It is a remarkable observation that in many cases, only a \uv{principal} part of the differential determines what the homology is.
This was exploited in \cite{HDA} and also e.g. in \cite{MMPB} to partially resolve the PROP for bialgebras.
Lemma \ref{lemmaA} is an application of this principle.
Here the minimality of $\oAres$ and the mild assumption $\omega(x,f)\in\oDres^{<n}$ (which in fact only formalizes what we mean by the principal part) are crucial for the spectral sequence argument to separate the principal part of $\dd$.
Apart from the minimality, arbitrary $\oAres$ with $X(0)=X(1)=0$ is allowed (unfortunately, this excludes e.g. unital algebras).
Notice, however, that we \emph{assume} that $\phi$ commutes with differentials.

To guarantee this, we need a stronger constraint on $\omega(x,f)$.
An easy sufficient way to ensure this is described in Lemma \ref{lemmaB}.
It leads to the definition \eqref{EQIdeal} of $\ideal{n}$.

Next, we have to construct a differential $\dd$ on $\oDres$ such that the assumptions of Lemma \ref{lemmaB} are satisfied.
This is achieved in Lemma \ref{LemmaDiffExists}.
To begin with, one obtains $\omega(x,f)\in\oDres^{<n}$ by an inductive argument on the arity of the generators from $X$ using Lemma \ref{lemmaA}.
Then we have to improve this result.
This is where the proof of Theorem~$7$ of \cite{HDA} is unclear.
We were not able to get the originally desired result $\omega(x,f)\in\ideal{n}_{\textrm{orig}}$ (recall \eqref{EQOrigIdeal}).
But if one is able to control the interplay between arity and degree of the generators from $X$ in a suitable way, one obtains at least $\omega(x,f)\in\ideal{n}$ by a simple degree count.
A sufficient control is achieved for the Koszul resolution of a Koszul operad with generating operations bound in a single arity and degree.
This is explained in Lemma \ref{KoszulGenerators}.
We note that Lemma \ref{LemmaDiffExists} can be proved under a weaker control over $X$, but the resulting conditions dont't seem to be of any practical interest.
\bigskip

Still, it might be possible to improve the proof of Lemma \ref{LemmaDiffExists} to get $\omega(x,f)\in\ideal{n}_{\textrm{orig}}$ even without the restrictions imposed on $\oA$, thus proving the original Conjecture $31$ of \cite{HDA}.
However, to our best knowledge, explicit examples of resolutions of diagrams $\oD=\oA_\ucat$ are known only for free categories $\ucat$ and for operads satisfying the assumptions of Theorem \ref{MAIN}.
Moreover, in these cases $\ideal{n} = \ideal{n}_{\textrm{orig}}$.
Hence these do not decide whether the conjecture is still plausible.

Notice a slightly stronger statement about what generators are needed to compose $\omega(x,f)$ can be made.
For example, if $\dg{f}=0$, then $\omega(x,f)$ lies in the ideal generated by $X_f(<n)$ in $\Fr{X_{\out{f}}(<n) \op X_{\inp{f}}(<n) \op \kspan{f}}$.
This can be deduced from the proof of Lemma \ref{LemmaDiffExists}.
However this doesn't seem to be important.

Finally notice that Lemma \ref{lemmaA} is already quite a big achievement - it reduces the problem of resolving $\oD$ to finding $\omega(x,f)$'s from $\oDres^{<\ar{x}}$ so that $\dd^2=0$ and the differential commutes with $\phi$.
Alternatively, by Lemma \ref{lemmaB}, the problem is reduced to finding $\omega(x,f)$'s from $\ideal{\ar{x}}$ so that $\dd^2=0$.

%%%%%%%%%%%%%%%%%%%%%%
%%%%%%%%%%%%%%%%%%%%%%
\section{\texorpdfstring{Bar-cobar resolution of $\oC$}{Bar-cobar resolution of C}} \label{SectionBC}
%%%%%%%%%%%%%%%%%%%%%%
%%%%%%%%%%%%%%%%%%%%%%

Now we make the content of Theorem \ref{MAIN} more explicit in the case $\oCres=\BC{\oC}$.
We apply Lemma \ref{EasyLemma} on the bar-cobar resolution $\oCres=\BC{\oC}$.
Denote
$$\Sigma^n := \set{ (\xleftarrow{f_n}\cdots\xleftarrow{f_1}) \in(\Mor\oC)^{\times n} }{\out{f_i}=\inp{f_{i+1}}\mbox{ for }1\leq i\leq n-1}$$
the set of chains of composable morphisms in $\ucat$ of length $n$, e.g. $\Sigma^1=\Mor\ucat$.
Denote $\Sigma:=\bigcup_{i\geq 1}\Sigma^i$.

Recall that the bar-cobar resolution $\BC{\oC}$ (e.g. \cite{V}, where the noncoloured case is treated - but the coloured case is completely analogous) is a quasi-free \colop{V} generated by \smod[V] $\kspan{\Sigma}$, where the degree of $\sigma\in\Sigma^{n}$ is $n-1$.
The derivation differential is given by
\begin{align*}
\dd (\xleftarrow{f_n}\cdots\xleftarrow{f_1}) &:= \sum_{i=1}^{n-1}(-1)^{i+n+1}(\xleftarrow{f_n}\cdots\xleftarrow{f_{i+1}})\oo(\xleftarrow{f_i}\cdots\xleftarrow{f_1})\ + \\
&\phantom{:=}+\sum_{i=1}^{n-1}(-1)^{n-i}(\xleftarrow{f_n}\cdots\xleftarrow{f_{i+1}f_i}\cdots\xleftarrow{f_1}).
\end{align*}
The projection $\phi_\oC:\BC{\oC}\to\BC^1{\oC}\cong\oC$ onto the sub \smod[V] of weight $1$ elements is a quism.

\begin{theorem} \label{BarCobarChiral}
Let $\chir{-}^\mathrm{NS}_2:\BC{\oC}\to\BC{\oC}\ot\BC{\oC}$ be a linear map satisfying $\chir{a\oo b}^\mathrm{NS}_2 = \chir{a}^\mathrm{NS}_2\oo\chir{b}^\mathrm{NS}_2$ for all $a,b\in\BC{\oC}$ and determined by its values on generators:
\begin{gather*}
\chir{(\xleftarrow{f_n}\cdots\xleftarrow{f_1})}^\mathrm{NS}_2 := \\
(\xleftarrow{f_n}\cdots\xleftarrow{f_1}) \ot (\xleftarrow{f_n\cdots f_1})\ +\\
+ \sum_{\substack{1\leq m\leq n-1\\ 1\leq j_1<\cdots<j_m\leq n-1}} (-1)^{\epsilon} (\xleftarrow{f_n}\cdots\xleftarrow{f_{j_m +1}}) \cdots (\xleftarrow{f_{j_1}}\cdots\xleftarrow{f_1}) \ot (\xleftarrow{f_n\cdots f_{j_m +1}}\cdots\xleftarrow{f_{j_1}\cdots f_1}),
\end{gather*}
where $\epsilon := mn+\frac{1}{2}m(m-1)+\sum_{i=1}^k j_i$.
Then $\chir{-}^\mathrm{NS}_2$ induces, via \eqref{EQDefChir} and \eqref{EQSigmaBalance}, the maps $\chir{-}^\mathrm{NS}_n$ and $\chir{-}_n$ of Lemma \ref{chiral}.
Moreover, $$(\id[i]\ot\chir{-}^{\textrm{NS}}_a\ot\id[b-i-1])\chir{-}^{\textrm{NS}}_b = \chir{-}^{\textrm{NS}}_{a+b-1}.$$
\end{theorem}

\begin{proof}
We apply Lemma \ref{EasyLemma}.
The only nontrivial properties to verify are $\dd\chir{-}^\mathrm{NS}_2 = \chir{\dd -}^\mathrm{NS}_2$ and $(\chir{-}^\mathrm{NS}_2\ot\id)\chir{-}^\mathrm{NS}_2 = (\id\ot\chir{-}^\mathrm{NS}_2)\chir{-}^\mathrm{NS}_2$.
This can be done directly, but it is annoying and doesn't explain the origin of $\chir{-}^\mathrm{NS}_2$.
Thus we go another way.
There is the following description of $\BC{\oC}$.
Let $$C_*(I):=\kspan{(\mathbf{0}),(\mathbf{1}),(\mathbf{01})}$$
be the simplicial chain complex of the interval, i.e. $\dg{(\mathbf{0})}=\dg{(\mathbf{1})}=0$, $\dg{(\mathbf{01})}=1$ and $\dd{(\mathbf{0})}=\dd{(\mathbf{1})}=0$, $\dd{(\mathbf{01})}=(\mathbf{1})-(\mathbf{0})$.
Then 
\begin{gather} \label{EQBCEquivalent}
\BC{\oC}=\frac{\bigoplus_{n\geq 0} \oC^{\oo(n+1)}\ot\otexp{C_*(I)}{n}}{M},
\end{gather}
where the subspace $M$ is spanned by 
\begin{gather*}
f_{n}\ot\cdots\ot f_{i+1}\ot f_i\ot\cdots\ot f_1\ot c_{n-1}\ot\cdots\ot c_{i+1}\ot(\mathbf{0})\ot c_{i-1}\ot\cdots\ot c_1\ +\\
- f_{n}\ot\cdots\ot f_{i+1}f_i\ot\cdots\ot f_1\ot c_{n-1}\ot\cdots\ot c_{i+1}\ot c_{i-1}\ot\cdots\ot c_1
\end{gather*}
for any $f_n,\ldots,f_1\in\oC$ of right colours and any $c_{n-1},\ldots,c_1\in C_*(I)$.
Let the grading and the differential $\dd$ on $\BC{\oC}$ be induced by $C_*(I)$ ($\oC$ is concentrated in degree $0$) in the standard way.
The operadic composition is defined by 
\begin{gather*}
(f_n\ot\cdots\ot f_1\ot c_{n-1}\ot\cdots\ot c_1)\oo (g_m\ot\cdots\ot g_1\ot d_{m-1}\ot\cdots\ot d_1) := \\
(f_n\ot\cdots\ot f_1\ot g_m\ot\cdots\ot g_1\ot c_{n-1}\ot\cdots\ot c_1\ot d_{m-1}\ot\cdots\ot d_1).
\end{gather*}
A dg \colop{V} isomorphism with the previous description is easily seen to be 
\begin{gather} \label{EQIsoBarCobar}
f_n\ot\cdots\ot f_1\ot c_{n-1}\ot\cdots\ot c_1 \mapsto (\xleftarrow{f_n}\cdots\xleftarrow{f_{j_m +1}})\cdots (\xleftarrow{f_{j_1}}\cdots\xleftarrow{f_1}),
\end{gather}
where $c_{j_m}=c_{j_{m-1}}=\ldots=c_{j_1}=(\mathbf{1})$ and all other $c_i$'s equal $(\mathbf{01})$ (remember we can get rid of $(\mathbf{0})$ using the defining relations).
The point is that $C_*(I)$ carries the obvious coassociative coproduct
$$\cm(\mathbf{0})=(\mathbf{0})\ot(\mathbf{0}),\quad \cm(\mathbf{1})=(\mathbf{1})\ot(\mathbf{1}),\quad \cm(\mathbf{01})=(\mathbf{0})\ot(\mathbf{01}) + (\mathbf{01})\ot(\mathbf{1})$$
and there is also the trivial coproduct on $\oC$ given by $\cm(c)=c\ot c$.
These induce coproduct on $\BC{\oC}$ by
$$\cm(f_n\ot\cdots\ot f_1\ot c_{n-1}\ot\cdots\ot c_1) := \left((\underbrace{\cm\ot\cdots\ot\cm}_{2n-1\textrm{ times}}) (f_n\ot\cdots\ot f_1\ot c_{n-1}\ot\cdots\ot c_1)\right) \cdot\tau,$$
where $\tau\in\Sigma_{4n-2}$ rearranges the factors in the expected way, which will be obvious from the following computation.
Denote $c^0:=(\mathbf{0})\ot(\mathbf{01})$, $c^1:=(\mathbf{01})\ot(\mathbf{1})$ and for the rest of the proof, let's order the factor of the tensor products from right to left, i.e. $(\mathbf{01})$ is in position $2$ in $c^1$ and $(\mathbf{1})$ is in position $1$.
Then
\begin{gather*}
\cm (f_n\ot\cdots\ot f_1\ot\underbrace{(\mathbf{01})\ot\cdots\ot (\mathbf{01})}_{n-1\textrm{ times}} ) = \\
=  \hspace{-2ex}\sum_{\substack{0\leq m\leq n-1\\ 1\leq j_1<\cdots<j_m\leq n-1}} \hspace{-1ex}\left( f_n\ot f_n\ot\cdots\ot f_1\ot f_1 \ot c^0\ot\cdots\ot\hspace{-3ex}\underbrace{c^1}_{\mathrm{position}\ j_m}\hspace{-2.75ex}\ot\cdots\ot\hspace{-2.75ex}\underbrace{c^1}_{\mathrm{position}\ j_1}\hspace{-2.5ex}\ot\cdots\ot c^0 \right)\cdot\tau,
\end{gather*}
where $c^1$ appears only at positions $j_1,\ldots,j_m$.
Applying $\tau$ and multiplying yields
\begin{align*}
\sum_{\substack{0\leq m\leq n-1\\ 1\leq j_1<\cdots<j_m\leq n-1}} (-1)^\epsilon &\left( f_n\ot\cdots\ot f_1\ot (\mathbf{0})\ot\cdots\ot \underbrace{(\mathbf{01})}_{j_m}\ot\cdots\ot\underbrace{(\mathbf{01})}_{j_1}\ot\cdots\ot (\mathbf{0}) \right) \ot \\
\ot &\left( f_n\ot\cdots\ot f_1\ot (\mathbf{01})\ot\cdots\ot \underbrace{(\mathbf{1})}_{j_m}\ot\cdots\ot\underbrace{(\mathbf{1})}_{j_1}\ot\cdots\ot (\mathbf{01}) \right),
\end{align*}
where $\epsilon = mn + \frac{1}{2}m(m-1) + \sum_{i=1}^m j_i$ comes from the Koszul convention.
This is exactly the claimed formula under the isomorphisms \eqref{EQIsoBarCobar}.

It is easily seen that $\cm(a\oo b)=\cm(a)\oo\cm(b)$.
$\cm$ is the coproduct induced on the quotient \eqref{EQBCEquivalent} by the tensor product of coassociative dg coalgebras $C_*(I)$ and $\oC$.
It is a standard fact that the tensor product is also a coassociative dg coalgebra, hence $\dd\cm=\cm\dd$, $(\cm\ot\id)\cm = (\id\ot\cm)\cm$.
Then $\chir{-}^\textrm{NS}_2:=\cm$ has the properties (C2)--(C6).

Originally, we found the coproduct of this lemma by hand. We are indebted to Benoit Fresse for suggesting its origin in $C_*(I)$.
\end{proof}

A completely explicit cofibrant resolution $\oDres$ of $\oD=\oA_\ucat$ gives rise to a cohomology theory for $\oA_\ucat$-algebras (i.e. $\ucat$-shaped diagrams of $\oA$-algebras) describing their deformations.
This is explained in \cite{IB}.
Unfortunately, the description of $\dd$ on $\oDres$ given in Theorem \ref{MAIN} is not even explicit enough to write down the codifferential $\codd$ on the corresponding deformation complex $\Der^*(\oDres,\oEnd_W)$, not to mention the rest of the $L_\infty$-structure.
For the basic example $\oA=\oAss$, we already proved in \cite{D} that $(\Der^*(\oDres,\oEnd_W),\codd)$ is isomorphic to the Gerstenhaber-Schack complex (see \cite{GS}) $(C^*_{\mathrm{GS}}(D,D), \codd_{\mathrm{GS}})$ (of a diagram $D$) for \emph{some} resolution $\oDres$.
The method, however, doesn't allow to find $\oDres$ explicitly.
We conjecture that this $\oDres$ has the form given by Theorem \ref{MAIN}:

\begin{conjecture}
In Theorem \ref{MAIN}, let $\oA:=\oAss$, let $\oAres:=\oAss_\infty$ be the minimal resolution of $\oAss$ and let $\oCres=\BC{\oC}$.
Then there are $\omega(x,f)$'s such that
$$(\Der^*(\oDres,\oEnd_W),\codd) \cong (C^*_{\mathrm{GS}}(D,D), \codd_{\mathrm{GS}}).$$
\end{conjecture}

Another very interesting problem is to find an operadic interpretation of Cohomology Comparison Theorem:
Recall that CCT, proved in \cite{GS}, is a theorem relating deformations of the diagram of associative algebras to deformations of a single associative algebra.
The point is that the deformations of the single algebra are described by Hochschild complex equipped with a dg Lie algebra structure given by Hochschild differential and Gerstenhaber bracket.
On the other hand, in known examples (see \cite{FMY}), the $L_\infty$-structure on operadic deformation complex of the diagram has nontrivial higher brackets (see \cite{FMY}).
This suggest that this $L_\infty$-algebra can be rectified to the one given by CCT.

%%%%%%%%%%%
%%%%%%%%%%%
%%%%%%%%%%%

\end{document}